\documentclass{svproc}
\usepackage{url}

\usepackage{hyperref}
\usepackage{type1cm}        % activate if the above 3 fonts are
\usepackage{makeidx}         % allows index generation
\usepackage{graphicx}        % standard LaTeX graphics tool
\usepackage{multicol}        % used for the two-column index
\usepackage[bottom]{footmisc}% places footnotes at page bottom

\usepackage{newtxtext}       %
\usepackage[varvw]{newtxmath}       % selects Times Roman as basic font

\usepackage{bm}% for bold math symbols
\usepackage{siunitx}% for table alignments

\newcounter{algorithm}%[section]

\newcommand{\bsa}{{\boldsymbol{a}}}

\newcommand{\bsk}{{\boldsymbol{k}}}

\newcommand{\bsx}{{\boldsymbol{x}}}

\newcommand{\bsz}{{\boldsymbol{z}}}

\newcommand{\bsX}{{\boldsymbol{X}}}

\newcommand{\bszero}{{\boldsymbol{0}}} % vector of zeros
\newcommand{\bsone}{{\boldsymbol{1}}}  % vector of ones
\newcommand{\bsgamma}{{\boldsymbol{\gamma}}}

\newcommand{\rd}{{\mathrm{d}}}

\newcommand{\bbE}{{\mathbb{E}}}

\newcommand{\bbR}{{\mathbb{R}}}
\newcommand{\bbS}{{\mathbb{S}}}

\newcommand{\bbU}{{\mathbb{U}}}

\DeclareSymbolFont{bbold}{U}{bbold}{m}{n}
\DeclareSymbolFontAlphabet{\mathbbold}{bbold}
\newcommand{\ind}{{\mathbbold{1}}}

\newcommand{\calF}{{\mathcal{F}}}

\DeclareMathOperator{\var}{Var}

\usepackage{paralist}
\usepackage{xcolor}

\begin{document}

% Todo: Bayesian numerical analysis, big crush
% Bahadur and Savage conditions
% using random $n$ we don't get unbiased

\mainmatter              % start of a contribution
\title{Error estimation for quasi-Monte Carlo}
\titlerunning{Error estimation for QMC}  % abbreviated title (for running head)
%                                     also used for the TOC unless
%                                     \toctitle is used
%
\author{Art B. Owen}%\inst{1}}
\authorrunning{A. B. Owen} % abbreviated author list (for running head)
%
%%%% list of authors for the TOC (use if author list has to be modified)
\tocauthor{Art Owen}
\institute{Stanford University, \\Stanford CA 94305, USA\\
\email{owen@stanford}\\ 
\texttt{https://artowen.su.domains}
}
\maketitle              % typeset the title of the contribution

\begin{abstract}
  Quasi-Monte Carlo sampling can attain far better
  accuracy than plain Monte Carlo sampling.
  However, with plain Monte Carlo sampling it is
  much easier to estimate the attained accuracy.
  This article describes methods old and new to
  quantify the error in quasi-Monte Carlo estimates.
  An important challenge in this setting is that
  the goal of getting accuracy conflicts with that
  of estimating the attained accuracy.  A related
  challenge is that rigorous uncertainty quantifications
  can be extremely conservative. A recent surprise
  is that some RQMC estimates have nearly symmetric
  distributions and that has the potential to allow
  confidence intervals that do not require either a central
  limit theorem or a consistent variance estimate.
% We would like to encourage you to list your keywords within
% the abstract section using the \keywords{...} command.
\end{abstract}

\section{Introduction}

Quasi-Monte Carlo (QMC) sampling is used to compute numerical
approximations to integrals. It is an alternative to plain
Monte Carlo (MC) sampling.  Under reasonable conditions,
QMC attains far better accuracy, at least asymptotically,
than MC does.
MC retains one advantage: it is easier to estimate the attained
accuracy of an MC estimate than it is for a QMC estimate.
We will see below how randomized QMC (RQMC) supports
error estimation, but even there, plain MC makes the task
easier.

Let $\mu$ be the integral of interest and
$\hat\mu_n$ be our estimate of $\mu$ using $n$
evaluations of the integrand $f$.
We want $\varepsilon_n=|\hat\mu_n-\mu|$ to be small
but we also want to know something about how
large this error is, and of course we cannot use $\mu$
to do that.  Estimating $\hat\mu-\mu$ or a distribution
for it is a form of uncertainty quantification (UQ).
In a setting with high stakes (think of safety-related
problems), we will want to have high confidence
that $\mu$ is in some interval $[a,b]$ all points of which
are acceptable. Then it is not enough for $|\hat\mu_n-\mu|$ to be small.
It should also be known to be small.
For some problems  we might only need to verify
that $a\leqslant \mu$ or that $\mu\leqslant b$.

An ideal form of UQ arises in the form of a `certificate': two
computable numbers $a\leqslant b$ where we are mathematically
certain that $a\leqslant \mu\leqslant b$.  Equivalently,
having an estimate $\hat\mu_n$ with
certainty that $|\hat\mu_n-\mu|\leqslant \Delta$ for known 
$\Delta<\infty$ provides a certificate.
These are rare for integration problems, but we will see
a few examples.  It is much more common for our
theorems to provide bounds for $|\hat\mu-\mu|$
that we cannot compute.
Existence of a finite error bound does not meaningfully quantify uncertainty.

We will look at several approaches to UQ
that all require some extra knowledge.  One approach
is to suppose that $f$ belongs to a given function
class $\calF$, in which we can bound $\varepsilon_n=\varepsilon_n(f)$.
Another is to use a model with random $\hat\mu_n$ 
where we can get bounds on the distribution of $\varepsilon_n$.
Sometimes we can only get asymptotically justified bounds
as the amount of computation grows to infinity.

We will assume that the method for getting a confidence
interval involves randomizing the points at
which $f$ is evaluated, via a pseudo-random
number generator.
This depends on using a well tested pseudo-random
number generator.  The big crush of \cite{lecu:sima:2007}
provides such a thorough test.
It is also possible to use a Bayesian model in which
$f$ itself is random.  The  most common choice is
a Gaussian process model for $f$.
For an entry point to that large literature, one can start with \cite{brio:etal:2019}.
The Bayesian approach can produce very accurate estimates
but there is no counterpart to the big crush for that approach,
and so  it does not as yet have a well tested objective
way to estimate error.
%We might instead suppose that $f$ itself is random,
%coming for instance from a Gaussian process model.
%Randomness from $f$ does not as yet have a comparable
%testing methodology.

Methods that estimate both $\mu$ and $\varepsilon_n$
have a tradeoff to make. This is most clear in
RQMC (see below) which uses $R$ random replicates
of an integration rule on $n$ function evaluations.
For a given cost $nR$ we get more accuracy in
$\hat\mu$ by taking larger $n$ and better UQ
by taking larger $R$ (hence smaller $n$).

A second tradeoff can arise.  
We may have certainty or just high confidence
 that $-\Delta_n\leqslant \hat\mu_n-\mu\leqslant \Delta_n$
where we also know that $|\hat\mu_n-\mu|/\Delta_n$
converges to zero, either deterministically or probabilistically.
Then our resulting estimate $\Delta_n$ of the size
of the error will for large enough $n$ 
be a significant exaggeration, by an unknown factor.
The bounds can be very conservative.

%Third issue: uncertainty about the UQ itself.

Since this article is to appear in an MCQMC proceedings,
we assume some background knowledge that would take too
much space to include here. For readers new to those
topics, here are some bibliographic references
that also include some history. Those topics include
discrepancy \cite{chen:sriv:trav:2014}, the ANOVA decomposition \cite[Appendix A]{practicalqmc}, weighted Hilbert spaces \cite{dick:kuo:sloa:2013}, 
digital QMC constructions \cite{dick:pill:2010} (such as Sobol' sequences
and Faure sequences) integration lattices \cite{sloa:joe:1994} and their 
randomizations \cite{lecu:lemi:2000}.
Some further references are given in context. Further background is
given on confidence interval methods as those are not widely
studied in the QMC literature.

The contents of this paper are as follows.
Section~\ref{sec:notation} introduces some notation along with basic
strategies to get an estimate $\hat\mu$ of $\mu$.
Section~\ref{sec:uq} describes the basics of uncertainty quantification.
There are certificates, confidence intervals, and asymptotic confidence
intervals.  The results making the more desirable claims are harder to apply.
Section~\ref{sec:certificates} gives some classical certificates (bracketing
inequalities) for numerical integration of convex functions along with a new one
for completely monotone integrands using QMC
points with nonnegative and nonpositive local discrepancy.
Section~\ref{sec:rqmc} covers RQMC. It emphasizes asymptotic
statistical confidence intervals as $n\to\infty$ and how hard
they are to apply to RQMC with a small number $R$ of independent
replicates of an RQMC rule on $n$ points.  An empirical investigation
in~\cite{ci4rqmc} showed that classical $t$-test intervals based
on $R$ independent replicates are surprisingly reliable and there
are now a few theoretical explanations to show why.
Section~\ref{sec:whqse} looks at a proposal by Tony Warnock
and John Halton that in the notation of this paper is like
using $R$ quasi-random replicates of a QMC rule instead of $R$ random replicates.
Their quasi-standard error has had some empirical success
and also has some failings. The idea might benefit from
a further investigation.
Section~\ref{sec:gail} describes the guaranteed automatic integration
library which aims to sample adaptively until a specified error
criterion $|\hat\mu-\mu|\leqslant\epsilon$ is met, either absolutely
or probabilistically.
Section~\ref{sec:new} presents areas of current active interest
in UQ for QMC: unbiased MCMC, normalizing flows, median of means
and results about $R$ growing with $n$.
Conclusions in Section~\ref{sec:conc} describe how
the present understanding of the problem has many gaps.
It appears that there will
be a role for the analysis of distributions that become
symmetric as $n\to\infty$ without becoming Gaussian.

\section{Notation and basic estimates of $\mu$}\label{sec:notation}

For integers $d\geqslant1$,
we use $1{:}d$ to represent the set $\{1,2,\dots,d\}$.
The cardinality of $u\subseteq1{:}d$ is denoted by $|u|$
and we use $-u=1{:}d\setminus u$.
For $\bsx\in[0,1]^d$ and $u\subseteq1{:}d$ we use
$\bsx_u\in[0,1]^{|u|}$ to represent the components of $\bsx$
with indices in $u$. For $\bsx,\bsz\in[0,1]^d$ we write
$\bsx_u{:}\bsz_{-u}$ for the point with $j$'th component
$x_j$ when $j\in u$ and $j$'th component $z_j$ when $j\not\in u$.
We abbreviate $\bsx_{-\{j\}}{:}\bsz_{\{j\}}$ to $\bsx_{-j}{:}\bsz_j$.
We use $\bsone = (1,1,\dots,1)\in\bbR^d$
and $\bszero = (0,0,\dots,0)\in \bbR^d$.
For $u\subseteq1{:}d$, we use $\partial^uf$ for the partial derivative
of $f$ taken once with respect to each $x_j$ for $j\in u$. By convention
$\partial^\varnothing f=f$.
Our estimate for the integral will be denoted by $\hat\mu$
or by $\hat\mu_n$ when it is important to emphasize
the sample size $n$.  When we have independent replicates,
the $r$'th one will be denoted by $\hat\mu^{(r)}$ or $\hat\mu_n^{(r)}$ as needed.

We say that $a_n=o(b_n)$ if $a_n/b_n\to0$ as $n\to\infty$.
We say that
$a_n = O(b_n)$ if there is finite $c$ with $|a_n|\leqslant c b_n$
for all but finitely many integers $n\geqslant1$.  
We say that $a_n = \Omega(b_n)$ if there
is $c>0$ with $a_n\geqslant c b_n$ for all but finitely many $n$.
We say that $a_n = \Theta(b_n)$ if $a_n = O(b_n)$ and $a_n = \Omega(b_n)$.

This paper surveys many different results from multiple literatures.
As a result, it is necessary for some symbols to be used with
more than one meaning, even as some symbols common in
the literature have been changed to reduce the number of conflicts here.
The specific use cases are different enough to be clear from context.

We focus on the problem of approximating
$$
\mu = \int_{[0,1]^d}f(\bsx)\,\rd\bsx
$$
for $d\geqslant1$.
Formulating a given problem this way 
may require extensive
use of transformations such as those in \cite{devr:1986}
in order to handle domains other than $[0,1]^d$
and distributions other than the uniform one.
Section~\ref{sec:new} points to some new work using normalizing
flows to augment those transformations.
For very smooth $f$ and small $d$, classical integration
methods from \cite{davrab} are quite effective at
estimating $\mu$. 

For larger $d$ and/or less regular
$f$, MC methods can be more accurate.
In plain MC we take $\bsx_i\stackrel{\mathrm{iid}}\sim\bbU[0,1]^d$ and then
\begin{align}\label{eq:muhatn}
\hat\mu_n =\frac1n\sum_{i=1}^nf(\bsx_i).
\end{align}
If $\mu$ exists then $\Pr(\lim_{n\to\infty}\hat\mu_n=\mu)=1$
by the strong law of large numbers. If
$$
\sigma^2 = \int_{[0,1]^d}(f(\bsx)-\mu)^2\,\rd\bsx<\infty
$$
then $\bbE( (\hat\mu_n-\mu)^2)=\sigma^2/n$
and the root mean squared error (RMSE) is $\sigma/\sqrt{n}$.
The $n^{-1/2}$ rate is not as good as the ones for smooth
low dimensional functions in \cite{devr:1986} but remarkably, it requires no
differentiability and is the same in any dimension.

In QMC, we replace uniform random $\bsx_i$ by carefully
chosen deterministic points $\bsx_i\in[0,1]^d$ before
estimating $\mu$ by $\hat\mu_n$ with the same
simple average~\eqref{eq:muhatn} that we use for MC.
These points are usually chosen by methods
that for large $n$,
have a small value for the star discrepancy
\begin{align*}
D_n^*&=D_n^*(\bsx_1,\dots,\bsx_n)
=\sup_{\bsa\in[0,1]^d}|\delta(\bsa)|,\quad\text{where}
\\
\delta(\bsa)&=\frac1n\sum_{i=1}^n\ind\{\bsx_i\in[\bszero,\bsa)\}-\prod_{j=1}^da_j
\end{align*}
is the local discrepancy at $\bsa$.
The QMC counterpart to the law of large numbers is that
if $f$ is Riemann integrable and $D_n^*\to0$ as $n\to\infty$,
then $\hat\mu_n\to\mu$. See \cite{nied:1978} which also
mentions a converse that if $\hat\mu_n\to\mu$ whenever
$D_n^*\to0$, then $f$ must be Riemann integrable.
We will see below that if $f$ is of bounded variation
in the sense of Hardy and Krause (written $f\in\mathrm{BVHK}$)
then $\varepsilon_n=O(D_n^*)$. Since it is possible to find
points with $D_n^*=O(n^{-1+\epsilon})$ for any $\epsilon>0$,
we $\hat\mu_n =O(n^{-1+\epsilon})$, and then we
can get asymptotically lower errors from QMC than
the RMSE of MC. 

The $\epsilon$ in that rate for $\varepsilon_n$  hides powers of $\log(n)$.
Those are present in some worst cases,
but they do not seem to appear in applications \cite{thelogs}.
If $f$ has a continuous
mixed partial derivative taken once with respect to each
of the $d$ components of $\bsx$, then we can show that $f\in\mathrm{BVHK}$
using a result from \cite{frec:1910}. That dominating mixed partial
derivative falls short of the complete $r$-fold differentiability
that appears in the $O(n^{-r/d})$ result mentioned earlier,
and the price we pay is in logarithmic factors.

We will also consider RQMC, a hybrid of MC and QMC.
In RQMC, the points $\bsx_1,\dots,\bsx_n$ are chosen
so that individually $\bsx_i\sim\bbU[0,1]^d$ while collectively
these points have low discrepancy:
for any $\epsilon>0$ and $B>0$,
$\Pr( D_n^*(\bsx_1,\dots,\bsx_n) \leqslant Bn^{-1+\epsilon})=1$
holds for all but finitely many integers $n\geqslant1$. 
We write this as $D_n^*=O(n^{-1+\epsilon})$.
For a survey of RQMC, see \cite{lecu:lemi:2002}.

\section{Uncertainty quantification}\label{sec:uq}

Uncertainty quantification takes things to the next level,
where we want to know something about $\mu_n-\mu$ or $|\mu_n-\mu|$.
In this section we review some well known UQ methods for integration methods.  
Perfect knowledge of $\mu_n-\mu$ implies perfect knowledge
of $\mu$ and then there is no more uncertainty to quantify. 
A known distribution for $\mu_n-\mu$ does not collapse that way.
Neither does perfect knowledge
of $|\hat\mu_n-\mu|$; it simply means that $\mu = \hat\mu_n\pm|\hat\mu_n-\mu|$.
Some methods provide information about the distribution of $|\hat\mu_n-\mu|$.
We start with the more desirable
UQ conclusions and then introduce some
compromises in order to get more usable methods.

By the Koksma-Hlawka inequality \cite{hick:2014},
\begin{align}\label{eq:koksmahlawka}
\varepsilon_n \leqslant D_n^*(\bsx_1,\dots,\bsx_n) V_{\mathrm{HK}}(f)
\end{align}
where $V_{\mathrm{HK}}(f)$ is the total variation of $f$
in the sense of Hardy and Krause.
See \cite{variation} for a definition and some properties
of $V_{\mathrm{HK}}$.  This bound has absolute certainty,
it applies to the specific integrand we have, the
value of $n$ that we used, and even the precise list
of points $\bsx_i$ that we used.

Equation~\eqref{eq:koksmahlawka} falls short of providing an
uncertainty quantification because we ordinarily do not know
either of the factors in it.  It is expensive to compute
$D_n^*$.  The commonly used algorithm of 
\cite{dobk:epps:mitc:1996} costs $O(n^{1+d/2})$. 
Let us suppose that it costs  $\Omega(n^{1+d/2})$.
Then in the time it would take us to compute $D_n^*$ for QMC
we could get $N=\Omega(n^{1+d/2})$ points in MC.
Then MC would get an RMSE of $O(n^{-1/2-d/4})$
in the time that QMC using known $D_n^*$ would get an error of $O(n^{-1+\epsilon})$.
We could in principle precompute $D_n^*$ for a number
of point sets, but then we would face the second
problem: $V_{\mathrm{HK}}(f)$  is extremely hard to compute,
far worse than $\mu$ itself.
For smooth enough $f$, it is a sum of $2^d-1$ integrals
of the absolute values of some mixed partial derivatives of $f$ \cite{variation}.
For less smooth $f$, the problem is even harder.
As a result, \eqref{eq:koksmahlawka} serves to show us
that QMC can be much better than MC, but it does not
generally let us verify that that has happened for given $f$ and $n$.

More modern bounds are derived for reproducing kernel
Hilbert spaces (RKHSs) of integrands.
The unanchored Sobolev spaces
described in \cite{dick:kuo:sloa:2013} are a prominent example.
For every $u\subseteq1{:}d$, define a weight $\gamma_u>0$
and then let $\bsgamma$ represent all $2^d$ of those weights.
%We might take $\gamma_{\varnothing}=\infty$ (getting a semi norm) but
%the other $\gamma_u$ are finite.
The squared norm in this space is
\begin{align}\label{eq:rkhsnorm}
\Vert f\Vert_{\bsgamma}^2 = \sum_{u\subseteq1{:}d}
\frac1{\gamma_u}
\int_{[0,1]^{|u|}}\biggl( \int_{[0,1]^{d-|u|}} \partial^uf(\bsx)\rd\bsx_{-u}\biggr)^2\rd\bsx_u,
\end{align}
for weak partial derivatives $\partial^u f$.  
The error $|\hat\mu-\mu|$ is upper bounded
by $\Vert f\Vert_{\bsgamma}$ times a corresponding $L_2$
discrepancy measure defined by a reproducing kernel \cite{dick:kuo:sloa:2013}.
The widely used product weights
have $\gamma_u=\prod_{j\in u}\gamma_j$ for $\gamma_j>0$.
When  $\gamma_j$ decays rapidly with increasing $j$,
for example with $\gamma_j = j^{-\eta}$ for $\eta>1$,
then strong tractability  (see \cite{sloa:wozn:1998}) holds,
in which the number of function evaluations required
to get an error below $\epsilon\Vert f\Vert_{\bsgamma}$
does not depend on the dimension $d$.
For $\eta>2$, strong tractability holds with
errors $O( n^{-1+\delta})$ for any $\delta>0$.
This tractability property has been widely used to design better
QMC point sets.  
However evaluating $\Vert f\Vert_{\bsgamma}$ for an integrand $f$ 
is usually quite difficult and so, just like the Koksma-Hlawka
inequality,  the error bounds in weighted
spaces do not provide an uncertainty quantification.

The MC setting is more favorable for UQ.
By the Chebyshev inequality we know that for $\lambda>1$,
\begin{align}\label{eq:chebineq}
\Pr\Bigl( |\hat\mu_n-\mu|\geqslant \frac{\lambda\sigma}{\sqrt{n}}\Bigr)\leqslant \frac1{\lambda^2}.
\end{align}
Equation~\eqref{eq:chebineq}
is only a probabilistic bound on $\varepsilon_n$, rather than
a certain bound like~\eqref{eq:koksmahlawka}.
However, the unknown quantity $\sigma$ in it is much
easier to work with than $V_{\mathrm{HK}}$ or $\Vert f\Vert_{\bsgamma}$ is.
In some settings we may know $\sigma$, or a bound for $\sigma$.
For example, if $0\leqslant f(\bsx)\leqslant 1$, then
$\sigma\leqslant 1/2$.  

%Equation~\eqref{eq:chebineq}  gives us a confidence interval for $\mu$.  
From the Chebyshev inequality we see that
\begin{align}\label{eq:chebci}
\Pr\Bigl( 
 \hat\mu_n -\frac{\lambda\sigma}{\sqrt{n}}
\leqslant\mu\leqslant
 \hat\mu_n +\frac{\lambda\sigma}{\sqrt{n}}
\Bigr)\geqslant 1-\frac1{\lambda^2}.
\end{align}
For known $\sigma$, equation~\eqref{eq:chebci}  describes a confidence interval:
a computable random interval with at least the desired
probability of containing the true value of $\mu$.
To have that probability be at least $99$\% we may
take $\lambda=10$.  
We have given up the certainty of our bound
in order to get a method based on $\sigma$ instead of $V_{\mathrm{HK}}$
or $\Vert f\Vert_{\bsgamma}$.
The Chebyshev inequality~\eqref{eq:chebineq}
is tight because for any $\lambda>1$
there is a distribution for $\varepsilon_n$  that makes it an equality.

The Chebyshev confidence interval is not commonly
used. %We may not know $\sigma$ or an upper bound for it.
Even when we have a bound for $\sigma$, the interval
can be very conservative.  As mentioned above, to get
99\% confidence, we can use $\lambda =10$.
A Gaussian random variable only has probability about $1.5\times 10^{-23}$
to be $10$ or more standard deviations from its mean.
Here the error probability bound of $0.01$ is quite loose, despite the
fact that it came from an inequality that is tight.

The usual UQ for MC is based on the central limit theorem (CLT) under which
\begin{align}\label{eq:clt}
\lim_{n\to\infty}\Pr\Bigl( \frac{\hat\mu_n-\mu}{\sigma/\sqrt{n}}\leqslant z\Bigr)=\Phi(z) 
\end{align}
where $\Phi$ is the cumulative distribution function of the
Gaussian distribution with mean $0$ and variance $1$.
In addition to being probabilistic, it is also asymptotic in $n$.
Its popularity stems from its wide applicability
after we estimate $\sigma$ from the same data we use for $\hat\mu$.

The CLT typically gives much narrower confidence intervals than
Chebyshev's inequality provides.
In plain MC with $n\geqslant2$ we can estimate $\sigma^2$ by
\begin{align}\label{eq:samplevar}
s^2=\frac1{n-1}\sum_{i=1}^n(f(\bsx_i)-\mu)^2.
\end{align}
The peculiar denominator $n-1$ makes $\bbE(s^2)=\sigma^2$.
If $\sigma^2<\infty$, then equation~\eqref{eq:clt} also holds with
$s$ replacing $\sigma$.
Then for $0<\alpha<1$,
\begin{align}\label{eq:cltci}
\lim_{n\to\infty}
\Pr\Bigl(
\hat \mu_n - \frac{s}{\sqrt{n}}\Phi^{-1}(1-\alpha/2)
\leqslant
\mu \leqslant \hat \mu_n + \frac{s}{\sqrt{n}}\Phi^{-1}(1-\alpha/2)
\Bigr)=1-\alpha.
\end{align}
Commonly chosen values for $\alpha$ are $0.05$ and $0.01$
which provide, respectively, asymptotic $95$\% and $99$\% confidence
intervals for $\mu$.  Here $\Phi^{-1}(0.975)\doteq 1.96$ and
$\Phi^{-1}(0.995)\doteq 2.58$.

The usual statistical confidence interval is a bit wider than
the one in~\eqref {eq:cltci} because it makes an adjustment
for the sampling uncertainty in $s^2$.  The interval in equation~\eqref{eq:cltci}
is not exact, even when $f(\bsx)$ has a Gaussian distribution.  We may correct this using 
\begin{align}\label{eq:stdci}
\lim_{n\to\infty}
\Pr\Bigl(
\hat \mu_n - \frac{s}{\sqrt{n}}t^{1-\alpha/2}_{(n-1)}
\leqslant
\mu \leqslant \hat \mu_n + \frac{s}{\sqrt{n}}
t^{1-\alpha/2}_{(n-1)}
\Bigr)=1-\alpha
\end{align}
where $t^{1-\alpha/2}_{(n-1)}$ is the $1-\alpha/2$ quantile
of Student's $t$ distribution on $n-1$ degrees of freedom.
The correction for degrees of freedom makes a meaningful difference
when $n$ is small which is valuable
for RQMC as we will see in Section~\ref{sec:rqmc} . The coverage is exactly $1-\alpha$
if $f(\bsx_i)$ are Gaussian (and $n\geqslant 2$) and the limit~\eqref{eq:stdci}
holds if $f(\bsx)$ has finite variance.

We would prefer a non-asymptotic and nonparametric confidence interval.
That would provide exactly $1-\alpha$ coverage probability without
requiring that $f(\bsx)$ have a distribution belonging to a finite dimensional
family, such as the Gaussian distributions.  
Unfortunately, exact nonparametric
confidence intervals do not exist under conditions
given in Bahadur and Savage \cite{baha:sava:1956}.  
They consider a set $\calF$ of distributions on $\bbR$.
Letting $\mu(F)$ be $\bbE(Y)$ when $Y\sim F\in\bbR$, their conditions are:
\begin{compactenum}[\quad (i)]
\item For all $F\in\calF$, $\mu(F)$ exists and is finite.
\item For all $m\in\bbR$ there is $F\in\calF$ with $\mu(F)=m$.
\item $\calF$ is convex: if $F,G\in\calF$ and $0<\pi<1$,
  then $\pi F+(1-\pi)G\in \calF$.
\end{compactenum}
%Then
Under these conditions, their Corollary 2 shows that a Borel set constructed
based on $Y_1,\dots,Y_N\stackrel{\mathrm{iid}}
\sim F$
that contains $\mu(F)$ with probability at least $1-\alpha$
also contains any other $m\in\bbR$ with probability at least $1-\alpha$.
More precisely: we can get a confidence set, but not a useful one.
They allow $N$ to be random so long as $\Pr(N<\infty)=1$.
%even if we also assume that $\bbE( |f(\bsx)|^r)<\infty$ for some 
%known $r<\infty$.

There are settings where non-asymptotic confidence intervals
can be constructed.  To get them, we need some extra knowledge
about the distribution of $f(\bsx)$ for $\bsx\sim\bbU[0,1]^d$
that does not come from sampling.
We saw above that knowing $\sigma$ allows us to get a conservative
confidence interval based on the Chebyshev inequality.

Suppose that independent $f(\bsx_i)$ satisfies known finite upper and lower bounds. Hoeffding's inequality \cite{hoef:1963} then provides confidence intervals
for $\mu$ in terms of those bounds.  In the MC context
it is natural to assume that the bounds are the same for all $i=1,\dots,n$.
Because the bounds are known we can make a linear adjustment
to $f$ and then without loss of generality
we may then suppose that
$0\leqslant f(\bsx)\leqslant 1$ for all $\bsx\in[0,1]^d$.
Then for $t>0$, Hoeffding's inequality gives
\begin{align}\label{eq:hoeffding}
\Pr( |\hat\mu-\mu|\geqslant t/n)\leqslant 2\exp(-t^2/n).
\end{align}
This holds for any distribution of $f(\bsx)$ with
$\Pr(0\leqslant f(\bsx)\leqslant1)=1$.  
It is possible because that set of bounded random variables does
not satisfy Bahadur and Savage's condition~(ii).

A very natural setting with known bounds arises
when $f(\bsx)$ takes the value $1$ for $\bsx\in A$
and $0$ otherwise. Then $\mu$ can be interpreted
as the probability that $\bsx\in A$ when $\bsx\sim\bbU[0,1]^d$. 
%Ordinarily, such $f\not\in\mathrm{BVHK}$
%unless the boundaries of $A$ are axis parallel \cite{variation}.  
%Despite this, some RQMC methods will give an RMSE of $o(n^{-1/2})$.

Because~\eqref{eq:hoeffding} holds for any 
random variable bounded between $0$ and $1$,
it is very conservative for some of them.  For instance
if $|f(\bsx)-1/2|\leqslant 1/2000$ always holds then the Hoeffding
interval is 1000 times wider than it has to be.
This extra width is a consequence of imperfect knowledge
that specifies a wider than necessary interval. 
Empirical Bernstein inequalities \cite{maur:pont:2009}
make use of the sample variance from~\eqref{eq:samplevar}
to get less conservative intervals that nonetheless
always have the desired coverage level.

There is ongoing work in forming confidence intervals
that have guaranteed coverage for bounded random
variables and finite $n$ while being asymptotically
as narrow as possible.
See \cite{aust:mack:2022} and \cite{waud:ramd:2024}.
Some of this work allows for dependence among the
values of $f(\bsx_i)$. Some of it provides confidence
statements that are always valid in that 
they produce intervals $[a_n,b_n]$ with
$\Pr( a_n \leqslant \mu\leqslant b_n,\forall n\geqslant1)\geqslant1-\alpha$.
These confidence sequences  have to be
wider at each $n$ than if we only require validity at one value of $n$.

When $0\leqslant f(\bsx)\leqslant 1$ holds then this also holds for an average
of $n$ evaluations of $f$ at some RQMC points.
Jain et al.\ \cite{jain:etal:2025} study confidence interval widths
using the empirical Bernstein and related betting methods from~\cite{waud:ramd:2024}
for RQMC using $R$ random replicates of $n$ RQMC points.
They impose a budget constraint that $nR=N$. If $f$ is regular enough for
RQMC to be very effective, then the optimal $n$ grows only very slowly
as $N\to\infty$. The resulting confidence intervals narrow at a faster
rate in $N$ than the ones based on Monte Carlo do, but they narrow
at a slower rate in $N$ than the RQMC standard deviations do.

%The random variable $Y$ is sub-Gaussian if
%$\Pr( |Y|\geqslant t ) \leqslant 2\exp( -ct^2)$ holds for
%some $c>0$.
% need to know $c$ and another expectation defining a norm

\section{Certificates and bracketing}\label{sec:certificates}

A certificate is a
(finite) computable number that is mathematically guaranteed
to be no smaller than $\varepsilon_n$.
This concept is widely used in 
convex optimization \cite{boyd:vand:2004}.
There the goal is to minimize a convex function $f$
over  $\bsx\in C$ for a convex set $C$.
Once $f$ has been evaluated at points $\bsx_i$ for $i=1,\dots,n$,
we can be sure that  the minimum is no larger than $\min( f(\bsx_1),\dots,f(\bsx_n))$.
This upper  bound does not require convexity of $f$.
If $f$ is convex, then with enough data, we can get a lower bound.
A convex function cannot  go below any of its supporting hyperplanes. 
Then $f$  has to be above the pointwise maximum of all its supporting hyperplanes
that the algorithm has encountered.
Now the minimum value of that pointwise maximum
provides a computable lower bound for the minimum of $f$.
See Chapter 5 of \cite{boyd:vand:2004} on duality. The user then has 
certainty that the answer lies between two computable numbers $a$
and $b$.  If we take $(a+b)/2$ as the estimate then we know
that the error is no larger than $(b-a)/2$.

There are a few instances in numerical integration where
we can get a certificate.  They are
known as bracketing inequalities.
Here we review those cases and then present a recent
quasi-Monte Carlo version based on nonnegative local discrepancy
and completely monotone functions.

It is obvious that if $f(\bsx)$ obeys an upper bound on a
set, then so does its average over that set.  Similarly for lower
bounds. 
Let $f$ be a nondecreasing function of $x$ over $[0,1]$.
Then $f((i-1)/n)\leqslant f(x)\leqslant f(i/n)$ for all $x\in[(i-1)/n,i/n]$
and $i=1,\dots,n$.
This allows us to bracket $\mu$ between the values of the left and
right endpoint rules:
\begin{align}\label{eq:bracketlr}
\hat\mu_{\mathrm{Left}}:= \frac1n\sum_{i=1}^nf\Bigl(\frac{i-1}n\Bigr)
\leqslant\mu\leqslant
\frac1n\sum_{i=1}^nf\Bigl(\frac{i}n\Bigr)
=:\hat\mu_{\mathrm{Right}}.
\end{align}
By raising the computation from $n$ to $n+1$ function evaluations
we get an interval for $\mu$.

A more interesting bracketing inequality holds when $f$
is a convex function on $[0,1]$.
Now the average of $f$ over $[(i-1)/n,i/n]$ 
is no smaller than $f((i-1/2)/n)$ by Jensen's inequality
and no larger than $f(((i-1)/n)+f(i/n))/2$ by a secant inequality.
As a result, $\hat\mu_{\mathrm{Mid}}\leqslant \mu\leqslant \hat\mu_{\mathrm{Trap}}$ where
\begin{align}\label{eq:bracketmt}
\hat\mu_{\mathrm{Mid}} & = 
\frac1n\sum_{i=1}^n f\Bigl( \frac{i-1/2}n\Bigr)\quad\text{and}
\quad \hat\mu_{\mathrm{Trap}}  = \frac1n\Bigl[
\frac12f(0)
+\sum_{i=1}^{n-1}f\Bigl(\frac{i}n
\Bigr)
+\frac12f(1)\Bigr]
\end{align}
are the midpoint and trapezoid rules, respectively.
This bracketing requires $2n+1$ function evaluations
because there is no overlap among the points used by the
two rules.

For smooth enough integrands, these error bounds from
bracketing inequalities are far wider than necessary \cite[Chapter 2]{davrab}.
If $f'$ is integrable, then the left and right endpoint rules 
have an error that
is asymptotically $O(1/n)$. If we average them
then we get the trapezoid rule with error $O(1/n^2)$
if $f'\in L_2[0,1]$ \cite{cruz2002sharp}.
%For continuous  $f''$ there

The same issue arises when we bracket the integrand
between the midpoint and trapezoid rules.
If $f''$ is continuous, then
$\hat\mu_{\mathrm{Mid}}= \mu - f''(z_{\mathrm{M}})/(24n^2)$
and $\hat\mu_{\mathrm{Trap}} = \mu + f''(z_{\mathrm{T}})/(12n^2)$
for points $z_{\mathrm{M}}, z_{\mathrm{T}}\in(0,1)$.
These oppositely signed errors can be largely canceled
by Simpson's rule. It satisfies
$\hat\mu_{\mathrm{Simp}} = (2\hat\mu_{\mathrm{Mid}}+\hat\mu_{\mathrm{Trap}})/3$.
If $f^{(4)}$ is
continuous on $[0,1]$ then by equation (2.2.6) of \cite{davrab}
$$
\hat\mu_{\mathrm{Simp}} = \mu  + \frac{f^{(4)}(z_{\mathrm{S}})}{180n^4}
$$
for some $z_{\mathrm{S}}\in(0,1)$. Note that this estimate uses $2n+1$
function evaluations.
Then for a convex function with four continuous derivatives
we get an estimate with error $O(n^{-4})$
known to lie within a computable interval of width $O(n^{-2})$.
It is hard to just barely bound an error.

There is a discussion of practical error estimation
in \cite[Chapter 4.9]{davrab}. Many of them require
special conditions on $f$.  One of the most applicable methods
is to get two estimates, such as $\hat\mu_n$ and $\hat\mu_{2n}$.
Then $|\hat\mu_{2n}-\hat\mu_n|$ is a rough estimate of the
error in $\hat\mu_n$.  It is also likely to be an overestimate of
the error in $\hat\mu_{2n}$ which will usually be a better
estimate than $\hat\mu_n$.  Despite its wide use, this
method can clearly give an unreliable estimate.
There is an analysis in \cite{lyne:1983} showing how
using the difference between two different integration
rules can be an unreliable way to quantify uncertainty.

For $d=1$ and a function with $r\geqslant1$ continuous derivatives
we can get estimates of $\mu$ with an error of $O(n^{-r})$
with an implied constant that depends on $\Vert f^{(r)}\Vert_\infty$.
For errors that decay like $n^{-r}$, an investigator may know
enough about their integrand to reason that $\Vert f^{(r)}\Vert_\infty$ cannot be extremely large compared to $n^r$ 
times the implied constant in an error bound.
Then, even lacking a known bound for $\Vert f^{(r)}\Vert_\infty$
they may be confident that some value of $n$ is good enough.
For example, they might be confident that
the error $|\hat\mu_n-\mu|$ is comparable in size
to a floating point error that they already find acceptable.

The situation is very different for multivariate problems.
Then there is a curse of dimension from \cite{bakh:1959}.
Here is a sketch based on \cite{dimo:2008}.
Suppose that the absolute value of the partial derivative of $f$ taken $a_j\geqslant0$
times with respect to component $x_j$
is never larger than some $M\in(0,\infty)$ for any $\bsx\in[0,1]^d$
when $\sum_{j=1}^da_j\leqslant r$.
Then there still exists $B>0$ such that
for any rule of the form $\sum_{i=1}^nw_if(\bsx_i)$
for $w_i\in\bbR$ and $\bsx_i\in[0,1]^d$
there is $f$ with $|\hat\mu(f)-\mu(f)|>Bn^{-r/d}$.
For randomized points $\bsx_i$, the RMSE
cannot be below $B'n^{-1/2-r/d}$ for some $B'>0$.
For large $d$, we cannot simply ignore the
integration error.  
It is also not easy to use methods tuned
for $r\geqslant 3$. Dimov \cite{dimo:2008} describes
them as `sophisticated'.
%With lower bounds that bad, upper bounds and
%estimates of upper bounds can only be worse
%for those function classes.

Multivariate certificates for integration are rare.
We can apply bounds~\eqref{eq:bracketlr}
and~\eqref{eq:bracketmt} to iterated integrals.
Then if $f$ is nondecreasing in each component
of $\bsx$ or if $f$ is convex, we get certificates
but only at rates $O(n^{-1/d})$ and $O(n^{-2/d})$
for estimates that  have errors
$O(n^{-2/d})$ and $O(n^{-4/d})$, respectively,
when $f$ is smooth enough.

Another use of convexity is given in \cite{gues:schm:2004}.
For a convex function $f$ defined on a simplex
$\bbS\subset\bbR^d$, with vertices at $\bsx_0,\bsx_1,\dots,\bsx_d$
$$
f\Bigl(\frac1{d+1}\sum_{i=0}^d\bsx_i\Bigr)
%f(\bsx_*)
\leqslant
\frac1{\mathrm{vol}(S)}\int_{S}f(\bsx)\rd\bsx \leqslant
\frac1{d+1}\sum_{i=0}^df(\bsx_i).
$$
%for $\bsx_* = \sum_{i=0}^d\bsx_i/(d+1)$.
This method can be applied to domains that are partitioned into
simplices, allowing us to generalize the rule in~\eqref{eq:bracketmt}
while reusing any function evaluations that are on a vertex of
more than one simplex.
For integration over $[0,1]^d$ partitioning this domain
into simplices would require at least the $2^d$ function evaluations on the
corners of the unit cube, and would therefore not scale well
to large dimensions.

It is possible to get a certificate for integrals over $[0,1]^d$
by generalizing equation~\eqref{eq:bracketlr}.
The integrand there is nondecreasing  and the input points
are `biased low' for the lower bound and `biased high'
for the upper bound. 
We can generalize the notion of points being biased low by
requiring that $\delta(\bsa)\geqslant 0$ holds for all
$\bsa\in[0,1]^d$. This property is known as
`nonnegative local discrepancy' (NNLD).  An analogous `nonpositive
local discrepancy' (NPLD) property has $\delta(\bsa)\leqslant 0$ for
all $\bsa\in[0,1]^d$.

The extension to general $d\geqslant 1$ requires a strong
monotonicity condition. In particular, it is not enough for
$f$ to be nondecreasing in each component of $\bsx$
individually with the other $d-1$ components held fixed.
We require $f$ to be completely monotone as described
by \cite{aist:dick:2015}.

The function $f$ is completely monotone on $[0,1]^d$
if
$$
\sum_{v\subseteq u}(-1)^{|u-v|}f(\bsx_{-v}{:}\bsz_v)\geqslant0
$$
holds for every nonempty $u\subseteq1{:}d$ and
all $\bsx,\bsz\in[0,1]^d$ with $\bsx\leqslant\bsz$ componentwise.
Taking $u = \{j\}$ this requires that
$f(\bsx_{-j}{:}\bsz_j)\geqslant f(\bsx)$ so $f$ is nondecreasing
in each of the $d$ variables.  For $u =\{j,k\}$ for $j\ne k$
we see that
$f(\bsx_{-j}{:}\bsz_j)-f(\bsx)$ is nondecreasing in $x_k$.
Generally an $r$-fold difference of differences is nonnegative.

If we use `increasing' as a less precise but more vivid term
for `nondecreasing', then we may say that $f$ is increasing in every
variable, the amount by which it is increasing in any variable
is increasing in any other variable, that amount in turn is increasing
in any third variable, and so on.
This is a strict requirement.  It is satisfied if $f$ is the cumulative
distribution function (CDF) of some random random vector $\bsX$, i.e., 
$f(\bsx) = \Pr( \bsX\leqslant \bsx\ \text{componentwise})$.
As a partial converse, if $f$ is also right continuous, then
\begin{align}\label{eq:compmono}
f(\bsx) = f(\bszero) +\lambda\nu([\bszero,\bsx])
\end{align}
holds for some probability measure $\nu$ and some $\lambda\geqslant 0$
\cite{aist:dick:2015}.

\begin{theorem}
Let $f$ be of the form in equation~\eqref{eq:compmono}.
If $\bsone-\bsx_1,\dots,\bsone-\bsx_n$ have NNLD, then
$$
\frac1n\sum_{i=1}^nf(\bsx_i)\geqslant \int_{[0,1]^d}f(\bsx)\rd\bsx.
$$
If $\bsx_1,\dots,\bsx_n$ have NPLD and either
all $\bsx_i\in[0,1)^d\cup\{\bsone\}$, or $\nu$ is absolutely
continuous with respect to Lebesgue measure, then
$$
\frac1n\sum_{i=1}^n f(\bsone-\bsx_i)\leqslant \int_{[0,1]^d}f(\bsx)\rd\bsx.
$$
\end{theorem}
\begin{proof}
This is Theorem 1 of \cite{gnew:krit:owen:pan:2024}.
\end{proof}

To get a certificate we can use $n$ points that have NNLD
and $n$ points that have NPLD. Gabai \cite{gaba:1967}
showed that the Hammersley sequence in $[0,1]^2$
(in any base $b\ge2$)
has NNLD. That is generalized in \cite{gnew:krit:owen:pan:2024}
to any digital net in $[0,1]^d$ where all of the generator matrices
are permutation matrices. They also show that some rank one
lattices have NNLD.  The proofs of those NNLD results are based
on the theory of associated random variables from
reliability theory \cite{esar:pros:walk:1967}. 
% published the same year as \cite{gaba:1967}.

NPLD points are harder to construct than NNLD
points. At first this is surprising.
If $x_1,\dots,x_n\in[0,1]$ have NNLD then
they `oversample' $[0,a)$ so they
undersample $[a,1]$.  Then $1-x_i$ undersample $[0,a]$ and therefore they undersample $[0,a)$.
As a result, $1-x_1,\dots,1-x_n\in[0,1]$ have NPLD.
For $d\geqslant2$ it is no longer true that NNLD points $\bsx_i$
must provide NPLD points $\bsone-\bsx_i$.  The root of the problem
is that extra points $\bsx_i\in[\bszero,\bsa)$ implies
fewer points $\bsx_i\in[0,1]^d\setminus[\bszero,\bsa)$.
But this set is not a hyperrectangle, and so neither is
$\{ \bsone-\bsx\mid \bsx \in [0,1]^d\setminus[\bszero,\bsa)\}$.
The NPLD property requires undersampling of hyperrectangles
containing the origin.
%NNLD constructions use points $\bsx_i$ where
%there is a positive association between $x_{ij}$
%and $x_{ij'}$ for $1\leqslant j<j'\leqslant d$.

There are two basic constructions for NPLD points in \cite{gnew:krit:owen:pan:2024}.
For $d=1$, the points $x_i = i/n$ for $i=1,\dots,n$ have NPLD.
For $d=2$ and Hammersley points $\bsx_i\in[0,1)^2$ the
points $\tilde\bsx_i = (\bsx_{i1}+1/n,1-\bsx_{i2})$ have NPLD.
This `shift-flip' transformation is from \cite{dick:krit:2006}.
Tensor products of NNLD point sets have NNLD and tensor
products of NPLD point sets have NPLD \cite{gnew:krit:owen:pan:2024}.

The greater challenge than finding NNLD and NPLD points
is the requirement that $f$ be completely monotone.
It can be weakened by finding a control variate function $g$
with $\int_{[0,1]^d}g(\bsx)\rd\bsx=0$ such that
$f+g$ is completely monotone, and then averaging $(f+g)(\bsx_i)$. 
This has not been explored
in the literature.  For some $f$ it would work to take
$g(\bsx) = c\prod_{j=1}^d(x_j-1/2)$ for large enough $c>0$.
Such a function $g$ is not `QMC friendly' because its ANOVA
decomposition has all of its variance in the highest-order
$d$-dimensional term.  As a result, for large $d$ we
expect $(1/n)\sum_{i=1}^ng(\bsx_i)$ to converge slowly
to $\mu$.  For instance, $g$ has a very large Sobolev
norm in the usual unanchored spaces. 
Using~\eqref{eq:rkhsnorm} for $\gamma_j=j^{-\eta}$
we get $\partial^ug=c\prod_{j\not\in u}(x_j-1/2)$ and then
\begin{align*}
\Vert  g\Vert_{\bsgamma}^2 
&= c^2\sum_{u\subseteq1{:}d}\frac1{\gamma_u}
\int_{[0,1]^{|u|}}\biggl(\int_{[0,1]^{d-|u|}}
\prod_{j\not\in u}(x_j-1/2)\rd\bsx_{-u}\biggr)^2\rd\bsx_u\\
&=c^2\sum_{u\subseteq1{:}d}1_{u=1{:}d} \prod_{j\in u}j^\eta\\
&=c^2(d!)^\eta.
\end{align*}
For $\eta =2+\varepsilon$ with $\varepsilon>0$, we get
strong tractability.
The number of function evaluations to
get a QMC error for $g$ below $\epsilon$
can be brought below $c(d!)^{1+\varepsilon/2}$ times a constant 
that depends on $\epsilon$
but does not depend on $d$. The resulting rate in $d$ is quite unfavorable
because $g$ has a very high norm in the weighted space.

Convergence rates for known NNLD and NPLD constructions
show a dimension effect. The best 
constructions for even $d$ in \cite{gnew:krit:owen:pan:2024} arise by taking
tensor products of two dimensional Hammersley point sets on $m$
inputs (for NNLD points) along with tensor products of Hammersley
point sets after the shift-flip transformation (for NPLD points).
For $f\in\mathrm{BVHK}$ the resulting rule on $n=m^{d/2}$
points provides an error of $O(n^{-2/d+\epsilon})$.
It is not known whether better convergence rates can be
attained by other NNLD and NPLD constructions.
For $d=2k+1$, one can either use a Cartesian product
of $k+1$ two dimensional sets (ignoring one component)
or use $k$ two dimensional sets and one endpoint rule.

While the $O(n^{-2/d+\epsilon})$  rate shows a strong dimension effect, it is better than
the rate for deterministic sampling of bounded functions that are 
simply nondecreasing in each
variable individually.  Deterministic methods cannot be better than
$O(n^{-1/d})$, and random methods cannot  have a better RMSE than
$O(n^{-1/d-1/2})$. See  \cite{papa:1993} for precise statements with proof.

The bracketing rules for convex functions that are twice continuously
differentiable provide guaranteed intervals of width $O(n^{-2/d})$
for $\mu$.  This is the optimal rate for integration of convex bounded
functions.
If $\calF=\{f:[0,1]^d\to[0,1]\mid\text{$f$ is convex}\}$ 
then there exists $c_d>0$ such that 
any deterministic integration rule has $\sup_{f\in\calF}|\hat\mu(f)-\mu(f)|
\geqslant c_dn^{-2/d}$ when
$n = (2m)^d/2$ for some integer $m\geqslant1$ \cite{kats:nova:petr:1996}.
The deterministic rule are allowed to be adaptive, selecting $\bsx_i$
based on $f(\bsx_{i'})$ for $i'<i$.
% Their  proof technique is to sandwich $f$ between the
% function $\underbar f$ that is zero everywhere and a
% function $\bar f_n\in\calF$
% that is zero at every sampled point $\bsx_i$. This $\bar f_n$
% is the pointwise supremum of all convex functions $f$
% with $f(\bsx_1)=\cdots=f(\bsx_n)=0$.
% Then $\hat\mu=0$ for either $\underbar f$ or $\bar f_n$
% but $\mu$ and hence the integration error can be as large as 
% $\int_{[0,1]^d}\bar f_n(\bsx)\rd\bsx$ which they study.

%They also show that a non-adaptive method can attain errors of $O(n^{-2/d})$
%without a smoothness condition.
%An adaptive Monte Carlo method can attain $\bbE(|\hat\mu-\mu|)=O(n^{-2/d-1/2})$
%and they show a non-adaptive Monte Carlo method with
%$\bbE(|\hat\mu-\mu|)=O(n^{-3/(2d)-1/2})$.

\section{Randomized quasi-Monte Carlo}\label{sec:rqmc}

In RQMC our points $\bsx_1,\dots,\bsx_n$ are individually $\bbU[0,1]^d$
while collectively having a small $D_n^*$.
We can make $R$ statistically independent replicates $\hat\mu^{(1)},\dots,\hat\mu^{(R)}$
of $\mu$ using such an RQMC procedure.
Then $\hat\mu_n^{(r)}=\hat\mu^{(r)}$ are IID with $\bbE(\hat\mu^{(r)})=\mu$.
If the RQMC points have $D_n^*=O(n^{-1+\epsilon})$ then
\begin{align}\label{eq:muhatrqmc}
\hat\mu  =\frac1R\sum_{r=1}^R\hat\mu^{(r)}
\end{align}
has RMSE $O(R^{-1/2}n^{-1+\epsilon})$.
If all of the mixed partial derivatives of $f$ taken at most
once with respect to each of the $d$ variables are
in $L_2[0,1]^d$, and we have scrambled some digital
nets via the algorithm from \cite{rtms}
or from \cite{mato:1998:2},
then $\hat\mu$ from~\eqref{eq:muhatrqmc}
has an RMSE that is $O(R^{-1/2}n^{-3/2+\epsilon})$.
Even higher order rates in $n$ are available from randomizations
of higher order digital nets \cite{dick:2011}.

These rates show that for a given number $nR$ of function
evaluations, we expect a better estimate $\hat\mu$
by taking larger $n$ and smaller $R$.
The RMSE of $\hat\mu$ describes the quality of our
estimate of $\mu$ but not the quality
of our UQ.  Taking a larger value of $R$ will give better
UQ, and this presents a fundamental tradeoff that
we don't have to consider in plain MC.
Here we consider fixed $R$ as $n\to\infty$.
Some comments on $R$ changing with $n$
are in Section~\ref{sec:new}.

We have two sample sizes to consider, $n$ and $R$.
The customary sampling results in statistics have been developed
for $n\to\infty$ independent observations.
In RQMC we study $n\to\infty$, but for the very dependent
observations used to create $\hat\mu_n^{(r)}$. Then there are
$R$ replicates.  While those are independent, we ordinarily
prefer small $R$ instead of $R\to\infty$.
The large sample size $n$ comes from
dependent data and the independent sample size $R$
is small, so we do not have the large number of independent estimates
that appears in most statistical theory.

We begin by presenting statistical results for $n\to\infty$
IID observations.  Those results provides only
a limited understanding of the accuracy of RQMC
confidence intervals. That understanding combined with some
knowledge of how RQMC works was used to design an extensive
empirical investigation in \cite{ci4rqmc}.
Those empirical results have been followed by
theoretical explanations of them and this is an
active area of research.

Two important issues are how the error in the CLT~\eqref{eq:clt}
decreases with $n$ and how quickly the coverage of the
standard Student's $t$ confidence interval~\eqref{eq:stdci}
approaches $1-\alpha$ as $n\to\infty$.
These are most often studied through moment quantities.
Let $\bar Y$ be the average of $n$ IID random variables $Y_i$
that have mean $\mu$, variance $\sigma^2>0$ and a finite third moment.
Then by the Berry-Esseen theorem, there exists a constant $C<\infty$
such that
$$
\sup_{-\infty <z<\infty}\,
\Bigl|\,\Pr\Bigl( \frac{\bar Y-\mu}{\sigma/\sqrt{n}} \leqslant z\Bigr)
-\Phi(z)\Bigr| \leqslant \frac{C\rho}{\sigma^3\sqrt{n}}
$$
for $\rho = \mathbb{E}(|Y-\mu|^3)$ and all $n\geqslant1$.
We can take $C=0.4748$ \cite{shev:2011}.
From this we see that the CLT takes hold at the $O(n^{-1/2})$
rate and that a scaled third central moment $\rho/\sigma^3$ governs the error.
Usable confidence intervals, like the standard one, also have
to contend with a generally unknown $\sigma$.

The error in confidence intervals like the standard one and
also some bootstrap confidence intervals
is commonly studied through scaled third and fourth moments
$$
\gamma = \frac{\mathbb{E}( (Y-\mu)^3)}{\sigma^3}
\quad\text{and}\quad
\kappa = \frac{\mathbb{E}( (Y-\mu)^4)}{\sigma^4} -3.
$$
These are known as the skewness and kurtosis of $Y$.
A Gaussian random variable has $\gamma=\kappa=0$.
If $\kappa<\infty$ and $Y_i$ are IID random variables with
the same distribution as $Y$ then $\hat\mu = \bar Y=(1/n)\sum_{i=1}^nY_i$
has skewness $\gamma/\sqrt{n}$ and kurtosis $\kappa/n$ \cite{beyondanova}.
Then the skewness of $\bar Y$ approaches $0$ more slowly than
the kurtosis does.  A rapidly vanishing kurtosis is consistent with the third moment
appearing in the Berry-Esseen bound but the fourth moment not appearing there.

The coverage error in the standard confidence interval for $\mu$ is
\begin{align}\label{eq:stdcoverr}
\Pr\Bigl(
\bar Y - \frac{s}{\sqrt{n}}t^{1-\alpha/2}_{(n-1)}
\leqslant
\mu \leqslant \bar Y + \frac{s}{\sqrt{n}}
t^{1-\alpha/2}_{(n-1)}
\Bigr) -(1-\alpha).
\end{align}
This error has been well studied by researchers in the 1980s,
especially Peter Hall \cite{hall:1986,hall:1988}, in
the context of bootstrap confidence intervals described below.
A very interesting finding is that for IID sampling
of $n$ observations, the coverage error in the standard interval is commonly
$O(1/n)$ which is better than the RMSE of $O(n^{-1/2})$ for $\hat\mu$.
The UQ achieves a better convergence rate than the estimate whose uncertainty it quantifies.
The one-sided coverage errors 
$\Pr\bigl(\mu\leqslant
\bar Y - {s}{t^{1-\alpha/2}_{(n-1)}/\sqrt{n}}\bigr) -\alpha/2$
and
$\Pr\bigl(\mu \geqslant \bar Y + {s}
t^{1-\alpha/2}_{(n-1)}/{\sqrt{n}}
\bigr) -\alpha/2$ both converge at the $O(n^{-1/2})$ rate
but a fortunate cancellation yields an $O(n^{-1})$ rate for~\eqref{eq:stdcoverr}.

For small values of $n$, some bootstrap confidence intervals
can work better than the standard interval.  The bootstrap is described  in
\cite{efrontibs}, \cite{davi:hink:1997} and \cite{hall:1992}
ranging from introductory to very technical.
A bootstrap sample $Y_1^*,\dots,Y_n^*$ is formed
by taking $Y_i^* = Y_{j(i)}$ for $i=1,\dots,n$ where
$j(i)\stackrel{\mathrm{iid}}\sim
 \bbU\{1,2,\dots,n\}$,
so $Y_i^*$ are sampled with replacement from $Y_1,\dots,Y_n$.
We can then compute $\bar Y^* = (1/n)\sum_{i=1}^n Y_i^*$.
That process can be repeated independently $B$ times ($B=1000$ is commonly
used) yielding $\bar Y^{*1},\dots,\bar Y^{*B}$.
We can sort those values getting
$\bar Y^{*(1)}\leqslant\bar Y^{*(2)}\leqslant\cdots\leqslant\bar Y^{*(B)}$.
The percentile confidence interval at the 95\% level is
$[\bar Y^{*(.025B)},\bar Y^{*(.975B)}]$. Nowhere does it use a
parametric distributional assumption, such as the Gaussian
distribution, for $Y_i$. It does require some moment assumptions
in order to be asymptotically correct.

The bootstrap $t$ confidence interval, also called the percentile $t$
confidence interval has some advantages over the ordinary
percentile interval.  The error is typically $O(1/n)$ for both
one-sided and two-sided intervals.  For a distribution with
$\gamma=\kappa=0$, the two-sided coverage error is $O(1/n^2)$,
so we expect it to work well for Gaussian or nearly Gaussian data.
In some simulations \cite{owen:smallci} 
varying $\gamma$, $\kappa$ and  $3\leqslant n\leqslant 20$,
the bootstrap $t$ 95\% confidence intervals had good coverage:
\begin{quotation}
``The bootstrap $t$ method is shown to be very effective for the construction of central
95\% confidence intervals for the mean of a small sample. Though it fails utterly when
$n = 3$, it achieves close to the nominal coverage for a diverse collection of continuous
sampling distributions provided $n\geqslant4$. 
The intervals can be quite long unless $n\geqslant6$, and
have a highly variable length unless $n\geqslant7$.''
\end{quotation}
Of the 7 continuous distributions in \cite{owen:smallci}, the bootstrap $t$
got close to nominal coverage in 6 of them. The exception was the
lognormal distribution, where the bootstrap $t$ got only about 90\%
coverage by $n=20$. Even that was noticeably higher than all the other
8 methods apart from an alternative formulation of the bootstrap $t$.

% From Hall (1988) p 929
%We argue that percentile-t does a better job than accelerated bias-cor-
%rection of getting third-order terms right, provided the variance estimate is
%chosen correctly.

The bootstrap $t$ method works as follows.
From each bootstrap sample, we compute
$t^*=\sqrt{n}(\bar Y^*-\bar Y)/s^*$ 
where $(s^*)^2=\sum_{i=1}^n(Y_i^*-\bar Y^*)^2/(n-1)$.
We do this $B$ times and
then sort the resulting values getting $t^{*(1)}\leqslant
t^{*(2)}\leqslant \cdots \leqslant t^{*(B)}$.
The interval uses the approximation
$$
\Pr\Bigl( 
 t^{*(0.025B)}\leqslant 
\sqrt{n}\frac{\bar Y-\mu}s \leqslant t^{*(0.975B)}\Bigr) \approx 0.95
$$
which leads to an approximate 95\% confidence interval for $\mu$ of the form
$$\Bigl[\bar Y- t^{*(0.975B)}\frac{s}{\sqrt{n}},\,
\bar Y- t^{*(0.025B)}\frac{s}{\sqrt{n}}\Bigr].$$

Hall \cite{hall:1988} gives some asymptotic expansions for the
coverage error in confidence intervals for the mean.
His Table 1 includes the percentile bootstrap, the bootstrap $t$
and the normal theory interval~\eqref{eq:cltci} that uses Gaussian
quantiles instead of $t$ quantiles.  
Surprisingly, he does not include the standard Student's $t$ intervals;
for those see~\cite{hallformula-tr}.
The coverage errors for two sided confidence intervals are
\begin{align*}
\text{Normal theory:} &\qquad(2/n) \varphi(z^{1-\alpha/2})\bigl[\phantom{-}0.14\kappa-2.12\gamma^2-3.35\bigr] +O(1/{n^2}),\\
\text{Student's $t$:} &\qquad(2/n) \varphi(z^{1-\alpha/2})\bigl[\phantom{-}0.14\kappa-2.12\gamma^2\,\,\,\phantom{-3.42}\bigr] +O(1/{n^2}),\\
\text{Percentile:} &\qquad(2/n)\varphi(z^{1-\alpha/2})\bigl[-0.72\kappa -0.37\gamma^2-3.35\bigr] +O(1/{n^2}),\quad\text{and}\\
\text{Bootstrap $t$:} &\qquad(2/n)\varphi(z^{1-\alpha/2})\bigl[-2.84\kappa+4.25\gamma^2\,\,\,\phantom{-3.42}\bigr] + O(1/{n^2}),
\end{align*}
where $\varphi$ is the probability density function of the standard Gaussian
distribution. 
Hall's names for the normal theory, percentile and bootstrap $t$
methods are, `Norm', `BACK' and `STUD'
respectively. 

Hall's assumptions include $8$ finite moments and a distribution
% his ABC is accelerated bias corrected $-2.63\kappa + 3.11\gamma^2-3.35$
for $Y_i$ whose support is not contained within an arithmetic sequence.
His table is for nominal coverage $1-2\alpha$ whereas we ordinarily
target coverage $1-\alpha$, but this only affects the scaling of the
lead term, and not the qualitative consequences of skewness and kurtosis.
Hall's formulas for $n=18$ were very close to the average coverage
for $16\leqslant n\leqslant 20$ in \cite[Table 6]{owen:smallci}, except for the
lognormal distribution.

The coverage error 
formulas show an advantage for the bootstrap $t$.  The term $4.25\gamma^2$
is positive, so it works to increase coverage.  It also has no intercept, while the
normal theory and percentile methods have a negative intercept
which works to decrease their coverage.
Simulations of coverage levels usually show that nonparametric
confidence intervals have a coverage level
that approaches the desired one from below as $n\to\infty$.
That is, undercoverage is more common than overcoverage.
Student's $t$ intervals have a positive coefficient for kurtosis,
a negative coefficient for squared skewness and no intercept.
%While Hall does not give a formula for the standard student based interval
%we know that it is exact for Gaussian data, and Gaussian
%data have $\gamma=\kappa=0$.
%We thus expect the corresponding formula to have a zero intercept, though not
%necessarily with the same coefficients for skewness and kurtosis
%that the normal theory interval has.

The distribution of the $t$ statistic under non-normality is much
studied.  For IID data from a distribution symmetric about $\mu$ with heavier
tails than the Gaussian, there is a tendency for confidence intervals
based on the $t$ statistic to be conservative, though precise statements
of this phenomenon require some extra steps.
An explanation and survey of this work, going back nearly 100 years, appears in \cite{cres:1980}.  At a high level, higher kurtosis in $Y_i$
brings a lower kurtosis in $t$, so that extreme thresholds
are exceeded by $|t|$ less often.
What happens is that very large values of $|Y_i-\mu|$ inflate the denominator
$s/\sqrt{n}$ more than the numerator $\bar Y-\mu$.
In the extreme, if we send $Y_1\to\pm\infty$ then $t\to \pm1$.  The consequence
is that for distributions with large kurtosis, the usual 95\% confidence
intervals, based on $|t|\leqslant t_{(n-1)}^{0.975}$ can cover the mean more often than the nominal level, because $t_{(n-1)}^{0.975}\geqslant 1.96$.

To switch from folklore to precise statements 
requires some additional care and caveats. Here are a few of those results.
One subtlety is that the tendency to overcoverage holds at customary confidence levels, 
but not necessarily at lower levels, such as those below 50\%. 
The distribution of $t$ does not depend on $\mu$
so it is studied with $\mu=0$.
Then the findings are mostly studied through 
$\tilde t = \tilde t_n= \sum_{i=1}^nY_i/(\sum_{i=1}^nY_i^2)^{1/2}$,
with $t=\tilde t(n-1)^{1/2}/(n-\tilde t^2)^{1/2}$, a monotone transformation.
%If $Y_i$ have kurtosis $\kappa$ then $s$ has kurtosis
%$-2n\bbE\bigl(\bigl(Y_1/(\sum_{i=1}^nY_i^2)^{1/2}\bigr)^4\bigr)$.
For even $\nu\geqslant 2$, $\bbE( \tilde t^\nu)\leqslant \bbE( Z^\nu)$
for $Z\sim\mathcal{N}(0,1)$; see Corollary 1 of  \cite{efro:1969}.
Corollary 2 shows that $\tilde t$ has negative kurtosis under conditions
that include IID symmetrically distributed $Y_i$.
If the random variables are distributed as a scale mixture of mean
zero Gaussians (i.e., $\mathcal{N}(0,\sigma^2)$ for random $\sigma$)
then the $t$ test is conservative at thresholds above $1.8$ \cite{benj:1983}.
More precisely: the bound was established for $2\leqslant n\leqslant 18$, 
with a critical threshold that never went above $1.8$ and decreased 
with increasing $n>6$. Most of the results are for symmetrically distributed
$Y_i$, but it is known that the skewness of $t$ 
is $-2\gamma/\sqrt{n} + O(n^{-3/2})$
where $\gamma$ is the skewness of $Y_i$ \cite[Equation (1.15)]{beyondanova}.
For a survey of related results, see \cite{shao:wang:2013}.

Now we switch to the RQMC context. We replace $n$ by 
our number $R$ of replicates
and $Y_i$ by $\hat\mu_n^{(r)}$.  Only a little is known about
the skewness and kurtosis of $\hat\mu_n$ for various
RQMC methods and much of that knowledge is recent.
Let these be $\gamma_n$ and $\kappa_n$, respectively.

First, there is an old result \cite{loh:2003} 
proving a CLT as $n\to\infty$ for the scrambled $(t,m,d)$-nets 
using the scrambling from \cite{rtms}
when the underlying digital net has $t=0$, under smoothness
conditions on $f$.  That setting then has $\gamma_n\to0$
and $\kappa_n\to0$ and we expect the standard intervals
should then have low coverage error.
The nets of Faure \cite{faures} have $t=0$
but the more widely used ones based on sequences
from Sobol' \cite{sobo:1967} only have $t=0$
for $d\leqslant 2$ \cite[Table 4.1]{nied92}.
On the other hand, randomly shifted lattice rules are known to
produce $\hat\mu$ that does not follow a CLT
\cite{lecu:mung:tuff:2010}. There the distribution of $\hat\mu$
has a density represented by a spline curve.

In the present context with $R$ replicates, we expect the coverage
error to be $O(1/R)$ as $R\to\infty$ for fixed $n$.
Those results don't tell us about what happens when $n\to\infty$
for fixed $R$.

The RQMC counterpart to our sample variance $s^2$
of~\eqref{eq:samplevar} is
\begin{align}\label{eq:rqmcsamplevar}
\tilde s^2 =\tilde s_n^2 
= \frac1{R-1}\sum_{r=1}^R(\hat\mu^{(r)}-\hat\mu)^2
\end{align}
and then the standard interval becomes
$\hat\mu \pm \tilde s t^{1-\alpha/2}_{(R-1)}.$
This interval would be exact if $\hat\mu^{(r)}$ had a Gaussian
distribution, and we can therefore expect it to work well
when a CLT applies to $\hat\mu_n^{(r)}$ as $n\to\infty$.  

The simulations in \cite{ci4rqmc} investigated 
three different confidence interval methods for RQMC:
the standard interval, the percentile bootstrap and
the bootstrap $t$.  The number of replicates was
$R\in\{5,10,20,30\}$ in keeping with the desire
to keep $R$ small.
There were 5 different RQMC methods.
Two of them were lattice rules from LatNet Builder \cite{lecu:mung:2026}
using random shifts modulo one, with and without
the baker transformation of \cite{hick:2002}.
The other three RQMC methods used Sobol's digital nets \cite{sobo:1967} 
using the direction numbers from \cite{joe:kuo:2008}.
They were randomized with either a digital shift (see \cite{lecu:lemi:2002}),
a matrix scramble of \cite{mato:1998:2} 
plus digital shift, or the nested uniform scramble from \cite{rtms}.
The sample sizes were $n=2^m$ for $m\in\{6,8,10,12,14\}$.
There were $6$ integrands with known $\mu$ chosen to mix cases that are
easy and difficult for RQMC, varying in their levels of smoothness
and varying in the extent to which they have low effective
dimension.    See~\cite{ci4rqmc} for a discussion.  An integrand that is easy to
integrate well is not necessarily one that makes the confidence
interval problem easy.   All of those integrands
were constructed to allow varying dimension $d$.
The simulation used $d\in\{4,8,16,32\}$.

Each confidence interval method was challenged with $2400$
use cases from $6$ integrands, $4$ dimensions, $5$ sample sizes,
$5$ RQMC methods and $4$ values of $R$.  Each challenge
was repeated $1000$ times by taking $R$ sample
values without replacement from a pool of $10{,}000$. The goal was to get
95\% coverage.  A method was deemed to fail if it would only
attain 94\% coverage.  If it covers $\mu$ less than
927 times out of 1000, that would happen with less than 0.04
probability (by the binomial distribution) for a method that
had coverage 0.94 or more. As a result any such setting
was deemed to be a confirmed failure of the RQMC confidence interval.

Of the 2400 cases, the percentile bootstrap was found to
fail 1689 times.  This is not surprising given the
results in \cite{owen:smallci}.
The bootstrap $t$ was found to fail 81
times. The standard interval failed only 3 times.
None of those failures were for $R=10$.
Figure 2 of \cite{ci4rqmc} erroneosly
shows one such failure. That happened because
the plotting command to set up the figure's axes
mistakenly did not use  type =``n" in the
call to the {\tt plot} function in the R language.

Based on the results in \cite{ci4rqmc}, the current
best recommendation for 95\% RQMC confidence intervals
is to use $R\geqslant 10$ independent replications
along with the standard Student's $t$ based confidence interval,
along with one's preferred RQMC method.
Next we relate this finding to the understanding of confidence
intervals from Hall's formulas.

From the data in \cite{ci4rqmc} it was possible to
compute sample values of the skewness and kurtosis in each
distribution of $\hat\mu$ based on $10{,}000$ evaluations
for each integrand, dimension $d$, sample size $n$ and RQMC method.
Inspection of those estimated values revealed that
most of the RQMC distributions had modest sample skewnesses
and many of them had very large sample kurtoses.  A
large kurtosis presents a difficulty for the bootstrap $t$
and an advantage for the standard interval as discussed
above.  The modest skewness removes a disadvantage
for the standard interval and it removes an advantage for the bootstrap
$t$. Figure 2 of \cite{ci4rqmc} shows that in some instances
with extremely large kurtosis, the standard confidence
interval had coverage well above the nominal 95\%
level.  This consequence of high kurtosis is
well known \cite{cres:1980} as mentioned above.

The empirical findings motivated 
subsequent work in \cite{lowskewness} which shows that for random linear
scrambling of a digital net in base $2$ with a random digital shift,  the
skewness of $\hat\mu_n$ is $\gamma_n=O(n^\epsilon)$ for any $\epsilon>0$,
so it is almost $O(1)$.  Furthermore, under a model with randomly
chosen generator matrices the skewness is $O(n^{-1/2+\epsilon})$.
The kurtosis of $\hat\mu$ from random linear scrambling
of a digital net in base $2$ is known to diverge to $\infty$
as $n\to\infty$ for an analytic function on $[0,1]$
by a finding in \cite[Section 3]{superpolyone}.
%That rules out a CLT for random linear scrambling
%even under smoothness conditions of \cite{loh:2003}.
The reasoning is as follows.  
There is an event of probability $\Omega(1/n)$ that gives a
squared error of $\Omega(1/n^2)$ (for smooth $f'$ that does
not integrate to zero).
That same event on its own contributes $\Omega(n^{-5})$
to the fourth power of the error.  The expected squared
error is $O(n^{-3+\epsilon})$ and so
$\kappa_n +3=\Omega(n^{-5})/O(n^{-6+2\epsilon})=\Omega(n^{1-2\epsilon})$.
An integrand on $[0,1]^d$ for $d>1$ with a smooth
contribution from a one dimensional main effect will have
the same diverging kurtosis and the critical event has
probability $\Omega(d/n^2)$. Having $\kappa_n\to\infty$ completely
rules out a CLT for $\hat\mu_n$ from matrix scrambling with a digital shift.

The empirical findings of modest skewness and potentially
very large kurtosis have not been established for the
other RQMC methods.  The other methods in \cite{ci4rqmc}
did not appear to have large skewness. Only three of 600 cases
had $|\hat\gamma|>4$. Those had large kurtoses and
since the variance of $\hat\gamma$ involves sixth moments,
they might just be sampling fluctuations.
Many of the
histograms of $\hat\mu$ had non-Gaussian but nearly symmetric
distributions.  
%There is an elegant orthant symmetry argument
%in \cite{efro:1969} showing conservativeness of the standard
%$t$ based confidence interval for symmetrically distributed
%data with or without high kurtosis.  Note that symmetry implies zero
%skewness (given a finite third absolute moment) but zero skewness
%does not imply symmetry.

The bias-corrected accelerated bootstrap (BCa) of \cite{efro:1987}
was not included in the simulations of either \cite{ci4rqmc} or \cite{owen:smallci}.
% his ABC is accelerated bias corrected $-2.63\kappa + 3.11\gamma^2-3.35$
Hall \cite{hall:1988} prefers the bootstrap $t$, while
DiCiccio and Efron \cite{dici:efro:1996}  
prefer the BCa.  Table 1 of \cite{hallformula-tr}
includes an entry `ABC'  for a method that is very close to BCa.  
The BCa has a coverage expression of
$-2.68\kappa + 3.17\gamma^2-3.42$ (after corrections).
The constant $-3.42$ and the $-2.68\kappa$ term both
suggest that BCa is not
well suited to the RQMC estimates that can have very large kurtosis
and minimal skewness.

\section{The Warnock-Halton quasi-standard error}\label{sec:whqse}

Tony Warnock and John Halton reasoned that QMC ideas should
also be usable to get not just an estimate of $\mu$ but also
an estimate of the uncertainty in an estimate of $\mu$.
Instead of using $R$ completely random estimates (by Monte Carlo)
they use QMC ideas to balance those $R$ replicates with the
goal of getting `better than random' replication.
Their proposal is in the 
technical reports \cite{warn:2001,warn:2002}
and article \cite{halt:2005}.

To get $R\geqslant2$ quasi-replicates of an integral
on $[0,1]^d$, Warnock \cite{warn:2001} takes QMC points $\bsx_i\in[0,1]^{dR}$.
Then for $r=1,\dots,R$ let
$$\tilde \bsx_{i,r} = 
( x_{i,d(r-1)+1}, x_{i,d(r-1)+2},\cdots, x_{i,dr})\in[0,1]^d$$
so $\bsx_i = (\tilde\bsx_{i,1},\tilde\bsx_{i,2},\cdots,\tilde\bsx_{i,R})$.
Now let
$$\hat\mu^{(r)} = \frac1n\sum_{i=1}^n
f(\tilde\bsx_{i,r})
$$
for $r=1,\dots, R$.  Then $\hat\mu = (1/R)\sum_{r=1}^R\hat\mu^{(r)}$
and its  quasi-standard error is
$$
\text{QSE} = \biggl(\frac1R\frac1{R-1}\sum_{r=1}^R(\hat\mu^{(r)}-\hat \mu)^2\biggr)^{1/2}.
$$
Halton \cite{halt:2005} reasons that when $\bsx_1,\dots,\bsx_n$
have very low discrepancy then
$\tilde \bsx_{i,r}$ is effectively independent of
$\tilde \bsx_{i,r'}$ for $1\leqslant r<r'\leqslant R$.
That holds for random $i\sim\bbU\{1,\dots,n\}$ with
any two values $r$ and $r'$ held fixed because $\bsx_1,\dots,\bsx_n$
have nearly the $\bbU[0,1]^{dR}$ distribution.
Then $f(\bsx_{i,r})$ and $f(\bsx_{i,r'})$ are effectively
independent of each other too.
This does not however make 
$\hat\mu^{(r)}$ and $\hat\mu^{(r')}$ effectively independent.
For that, it would suffice to have
$(\tilde\bsx_{1,r},\dots,\tilde\bsx_{n,r})$
effectively independent of 
$(\tilde\bsx_{1,r'},\dots,\tilde\bsx_{n,r'})$, but
small $D_n^*$ does not imply this.

The QSE can be far too small~\cite{warnockhaltonqse}.
For instance if $\bsx_i$ are points of a Sobol' sequence
and $f$ is additive, then we will get a QSE of zero. That happens 
because $\{x_{1j},x_{2j},\dots,x_{nj}\} = \{0,1/n,\dots,(n-1)/n\}$
for all $j=1,\dots,dR$ and then $\hat\mu^{(r)}=\hat\mu^{(1)}$
for $r=2,\dots,R$. Some good results forming confidence intervals
based on the QSE are reported in \cite{warn:2002}. That source
may be hard to find; the results are described in \cite{warnockhaltonqse}.

Ideas like the QSE have potential to bring QMC accuracy 
with respect to $R\to\infty$ for uncertainty quantification 
based on $R$ replicates instead of
the customary MC rate in $R$.  The near independence that
Halton writes about could perhaps be achieved using
a QMC point set $\tilde\bsX\in\bbR^{R\times nd}$.
We can rearrange row $r$ of $\tilde\bsX$ into an $n\times d$ matrix
to use for the $r$'th QMC replicate.
Finding a good QMC point set with $R$ rows and $nd$ columns
presents a substantial challenge for the values of $n$, $R$ and $d$
commonly used in RQMC.
There may however be some alternative way to get QMC
accuracy with respect to $R$ using something more
complicated than their procedure that is also less cumbersome
than finding $R$ QMC points in dimension $nd$.
Their proposal has not been explored or modified very much.

\section{Guaranteed automatic integration library}\label{sec:gail}

Most of the uncertainty quantification methods in this article are
about finding some value $\epsilon$ where, after sampling $f$
$n$ times, we have reason to believe that $|\hat\mu_n-\mu|\leqslant\epsilon$.
The evidence behind that belief might be a certificate, a confidence interval
or an asymptotic confidence interval.
A complementary approach is to start with a target value of $\epsilon$
and look for a value of $n$ for which we will have 
reason to believe that $|\hat\mu_n-\mu|\leqslant\epsilon$.
There are many results in the literature showing how $n$
must grow asymptotically as $\epsilon$ is reduced.
The Guaranteed Automatic Integration Library (GAIL) \cite{tong:etal:2022}
may be the unique one that provides a non-asymptotic solution
for integration over $[0,1]^d$.  The GAIL project continues
as part of the {\text{qmcpy}} Python library
\cite{choi:etal:2024}.

In GAIL, the user specifies $\epsilon$ and then the library is
designed to return an interval $[a,b]$ of width no more than
$2\epsilon$ with
$a\leqslant \mu\leqslant b$. 
In other words, GAIL provides a bracketing solution where
the user can specify the width of the bracket instead of specifying $n$.
A probabilistic version delivers an interval
with $\Pr( a\leqslant \mu\leqslant b)\geqslant 1-\alpha$.

A first approach to GAIL using MC is in \cite{gailmc}.
As noted in Section~\ref{sec:uq} such an approach requires some
extra knowledge about $f(\bsx)$. The approach in \cite{gailmc} is to use
a preliminary sample to get a probabilistic upper bound for $\sigma$.
If $\Pr( \sigma \leqslant \hat\sigma) \geqslant  1-\alpha_1$ for $\alpha_1<\alpha$
then we only need to find $a<b$ with
$\Pr( a \leqslant \bbE( f(\bsx))\leqslant b)\geqslant 1-\alpha_2$
that holds whenever $\var( f(\bsx))\leqslant\hat\sigma^2$ and
$\alpha_1+\alpha_2\leqslant \alpha$.

Now we need an upper confidence limit for $\sigma$.
This seems like it should be even harder to get than
the confidence interval for $\mu$ that we set out to get
and it also requires extra information.
The extra information used in~\cite{gailmc} is the
assumption that
\begin{align}\label{eq:kurtosisassumption}
\bbE( (f(\bsx)-\mu)^4 ) \leqslant \tilde\kappa\sigma^4
\end{align}
for known $\tilde\kappa<\infty$.
That assumption allows for a first Monte Carlo 
sample with $n_1\geqslant 2$ function evaluations
to give an upper confidence bound $\hat\sigma^2$
for $\sigma^2$. That is
followed by a second independent Monte Carlo sample
to provides a confidence interval for $\mu$. The second sample
size $n_2$  can now be chosen using 
$\hat\sigma^2$ to get the desired
interval length.  The first stage uses a probability inequality of
Cantelli and the second uses a Berry-Esseen inequality
where the third absolute moment $\rho$ is bounded
using $\tilde\kappa$.

The assumption~\eqref{eq:kurtosisassumption} is discussed
in \cite{gailmc}. It corresponds to a cone of integrands
\begin{align}\label{eq:cone}
\bigl\{ f\in L_4[0,1]^d \mid \Vert f-\mu(f)\Vert_4 
\leqslant \tilde\kappa^{1/4}\Vert f-\mu(f)\Vert_2\bigr\}.
\end{align}
Scaling $f$ by a constant factor keeps it within the cone.  That is quite
different from using a ball of functions instead, because
scaling a function can move it in or out of a ball.
The cone~\eqref{eq:cone} is non-convex, so it avoids condition (iii)
in the impossibility result of Bahadur and Savage. 
The cost of the algorithm
is not very sensitive to overestimation of $\tilde\kappa$.
If $f(\bsx)$ has a Gaussian distribution with $\sigma>0$
then $\tilde\kappa=3$.  No distribution has
$\tilde\kappa<1$.

A QMC version of GAIL  is in \cite{hick:ruga:2016} .
It is based on the Sobol' sequence and a Walsh
decomposition of the integrand.  See \cite{dick:pill:2010}
for both of those.
The Walsh decomposition expands $f$ into a
sum of Walsh functions indexed by $\bsk\in\{0,1,2,\dots\}^d$
and multiplied by coefficients $\hat f_{\bsk}$.
%The indices $\bsk$ are placed into a linear ordering
%where indices nearer $(0,0,\dots,0)$ tend to precede
%indices farther from $(0,0,\dots,0)$.  There is more
%than one ordering consistent with their method.
The integration error can be bounded by a sum of
$|\hat f_{\bsk}|$ over certain indices $\bsk$ that
depend on which Sobol' points are used.
Given a sample from a Sobol' sequence it is possible to estimate
some of the Walsh function coefficients.
They make an assumption about how Walsh coefficients
decay as $\bsk$ moves farther from $\boldsymbol{0}$.
Their assumption goes beyond the decay of Walsh
coefficients noted by \cite{dick:2009} and \cite{yosh:2017}
for smooth $f$.
That assumption places $f$ inside a cone
and it lets them bound the
sum of the remaining absolute coefficients given the 
estimated sum of absolute coefficients.  
They keep doubling the number of Sobol' points
and updating the error bound until the error bound is small enough. 

A counterpart for rank one lattices, using a Fourier decomposition
is in \cite{jime:hick:2016}.  A discussion of relative error is studied in \cite{sorokin2022bounding}.

\section{New directions}\label{sec:new}

We cannot always write an expectation $\mu$ of interest as
an integral of some computable function $f$ over $[0,1]^d$ for finite $d$,
even with all the methods of \cite{devr:1986} at our disposal.
For instance, many problems in Bayesian computation have
$\mu = \bbE( g(\bsx))$ for $\bsx\sim p$ where we cannot readily
sample independently from the posterior distribution $p$.
This $p$ may depend on arbitrary details of a very large data set.
This sampling difficulty is the main impetus for methods like Markov
chain Monte Carlo (MCMC) \cite{broo:gelm:jone:meng:2011}
and  particle filters \cite{chop:papa:2020}.
Because MCMC and particle methods sample dependently it
is more difficult to quantify the uncertainty in their estimates.

\subsection*{Unbiased MCQMC}
A second challenge with UQ for MCMC is that their estimates typically
have a bias.  That bias may disappear exponentially fast with
$n$ while remaining large in practice because $A\rho^n$ for $\rho$
just barely smaller than one may be large for the $n$ we can use.
It is much easier to quantify uncertainty in unbiased estimates that
can be independently replicated.  
Coupling from the past (CFTP) \cite{prop:wils:1996} can generate unbiased
estimates in some MCMC settings, but it only works well in very
restrictive settings.
There is some recent work on more generally applicable
coupling strategies that remove the bias from MCMC.  See
\cite{atch:jaco:2025} for a discussion of those methods.
That allows independent replicates of unbiased MCMC methods
to be used for UQ.

It is possible to embed QMC and RQMC methods into MCMC
by replacing the sequence of IID
random numbers driving the MCMC sampling by a sequence
that is completely uniformly distributed (CUD) \cite{chen:1967,chen:2011,trib:2007,qmcmetro,qmcmetrocts,liu:2024}.
Using a CUD sequence is like using a small pseudo-random number
generator in its entirety.
CUD sequences are described in \cite{levi:1999} and a probabilistic
version called weakly CUD is given in \cite{qmcmetro2}.
These methods are proven to converge to $\mu$ under assumptions
similar to those where usual MCMC converges. Faster convergence is usually
observed especially for the Gibbs sampler. Chen \cite{chen:2011} 
establishes an RMSE of $o(n^{-1/2})$, under strong assumptions.
These methods can be replicated but as for IID MCMC sampling, replication
does not let us eliminate the effects of bias.

By combining unbiased MCMC with a weakly CUD driving sequence,
\cite{du:he:2024}  are able to get independent
unbiased replicates of MCMC estimates for the Gibbs sampler.
That allows simple replication for MCMC with empirically better
convergence than by plain MCMC.  Much of the analysis extends
beyond the Gibbs sampler, but the Gibbs sampler is smooth
which can help it benefit from more evenly distributed inputs
and their algorithm also uses a coupling strategy designed for
the Gibbs sampler.
When CFTP is applicable, it can be used with RQMC \cite{lecu:sanv:2010}.

\subsection*{Normalizing flows}
There has been much recent interest in normalizing
flows \cite{reze:moha:2015}
that produce a transformation $\phi$
of $\bsx\sim\bbU[0,1]^d$ so that $\bsz=\phi(\bsx)\sim q$ where
$q\approx p$.
Usually $\phi$ is a transformation of $d$ Gaussian random variables, but
those can be expressed as a transformation of $d$ uniform ones.
Then we may estimate $\mu$ by a self-normalized importance sampling
estimate
$$
\hat\mu = \frac1n\sum_{i=1}^n g(\bsz_i)\frac{p(\bsz_i)}{q(\bsz_i)}
\Bigm/
\frac1n\sum_{i=1}^n \frac{p(\bsz_i)}{q(\bsz_i)}
$$
where $\bsz_i = \phi(\bsx_i)$ for $\bsx_i\sim\bbU[0,1]^d$.
We need to be able to compute $p$ up to a normalizing constant.
We need to sample from $q$ and also evaluate an unnormalized
version of it, but we are free to choose $q$ from a flexible
parametric family for which these are feasible.
It is also required that $q(\bsz)>0$ whenever $p(\bsz)>0$.

The use of normalizing flows allows MC and RQMC methods
to compute confidence intervals unaffected by the bias in MCMC.
The first effort using RQMC in normalizing flows is in \cite{andr:2024}.
The results there show decreasing effectiveness as the dimension
increases.  It is reasonable to suppose that the ratio $p/q$ becomes
very unfavorable to RQMC as the dimension increases.  
Some more favorable results are in \cite{liu:2024:transport}
which uses transformations
tuned to RQMC along with some dimension reduction ideas.

\subsection*{Median of means}

The distribution of RQMC estimates does not generally follow
a CLT.  When the matrix scramble of \cite{mato:1998:2} 
with a digital shift is applied to a Sobol' net,
the distribution of $\hat\mu_n-\mu$ can be
very non-Gaussian for a smooth integrand.
As noted above, the kurtosis of $\hat\mu_n$ may diverge to infinity
as $n$ increases, while the skewness remains modest or even converges to zero.
In that setting, much of the variance is due to rare events of
probability $\Omega(1/n)$ where $|\hat\mu-\mu|=\Omega( n^{-1})$.
If one takes the median of $R$ replicates instead of the mean,
then the outliers among $\hat\mu^{(1)},\dots,\hat\mu^{(R)}$
are essentially ignored.
This is called a `median of means' because each $\hat\mu^{(r)}$
is the mean of $n$ function evaluations, and we then
take $\hat\mu = \mathrm{median}(\hat\mu^{(1)},\dots,\hat\mu^{(R)})$.

For analytic functions on $[0,1]^d$, a median-of-means
approach brings an error of $O( n^{-c\log_2(n)/d})$ for any $c<3\log(2)/\pi^2\approx 0.21$ \cite{superpolymulti}.
This rate is called super-polynomial because it is better than $O(n^{-r})$ for any
finite $r$.  A median-of-means strategy adapts to a possibly
unknown level of smoothness in $f$. For $R\geqslant \log_2(n)$, 
it attains an RMSE of
$O( n^{-\alpha-1/2+\epsilon})$ for any $\epsilon>0$
when $f$ has finite variation of order $\alpha$ (defined in \cite{dick:2011}).
This is proved in Theorem 2 of \cite{pan:2024:tr}.

There are some other applications of the median of means
method that attain a universal goodness property.
Here are sketches of two of them; the reader should
read them to get the full details.
A median of means approach to RQMC by lattice
rules has been developed by \cite{goda:lecu:2022}.
Lattice rules can be tuned to specific Hilbert space weights.
% but we don't want to infer that they did that
The median rules in~\cite{goda:lecu:2022} can be constructed
without specifying those weights and yet they
attain nearly the optimal worst-case error rate for any reasonable
choice of weights and smoothness.
Goda and Krieg \cite{goda:krie:2024:tr} choose a
random prime number $p$ from the interval $[\lfloor n/2\rfloor+1,n]$.
Then they choose a rank one lattice in $[0,1]^d$ with $p$
points using a generating vector chosen uniformly from $\{1,2,\dots,p-1\}^d$.
From that rank one lattice, they estimate $\mu$.
They repeat this $R=\Omega( \log(n))$ times  independently and
take $\hat\mu  =\mathrm{median}(\hat\mu^{(1)},\dots,\hat\mu^{(R)})$.
The resulting $\hat\mu$ is nearly optimal in any Korobov class of 
periodic functions with smoothness $\alpha>1/2$. 

The great promise of median-of-means estimates strongly
motivates us to seek confidence intervals or some other
uncertainty quantification for them.
Suppose that $\hat\mu_r$ has a continuous distribution
with a unique median $\tilde\mu$.
Then $\Pr( \hat\mu_r <\tilde\mu)=\Pr( \hat\mu_r >\tilde\mu)=1/2$.
It is straightforward to sort the $\hat\mu^{(r)}$ into
$\hat\mu^{[1]}\leqslant \hat\mu^{[2]} \leqslant \cdots\leqslant \hat\mu^{[R]}$
and from that get a confidence interval on $\tilde \mu$.
For instance
$\Pr( \tilde\mu < \hat\mu^{[r]}) = \sum_{\ell=0}^{r-1}{R\choose \ell}/2^R$
can be used to get an upper confidence limit for $\tilde\mu$
(depending on $r$)
and a lower limit can be attained similarly.
This falls short of providing a confidence interval for $\mu$
because there is no assurance that $\tilde\mu=\mu$.
We would need a computable bound  (probabilistic or otherwise)
for $|\tilde\mu-\mu|$ in order to get a non-asymptotic
uncertainty quantification for $\mu$ this way.

The median of means method is most commonly
used on $n$ IID random variables $Y_i$ with finite variance $\sigma^2$.
The emphasis is very different from our use taking the median of $R$
independent estimates that each use $n$ very dependent values.
While a CLT gives asymptotic confidence intervals for $\mu=\bbE(Y_i)$
the goal in much median of means research is to get 
the desired coverage for finite $n$. That
requires somewhat wider confidence interval than the CLT based ones and 
also is only available for $\alpha\geqslant \alpha_{\min}>0$ 
for a threshold $\alpha_{\min}$ that depends on some assumptions
about the distribution of $Y_i$. 
See \cite{devr:lera:lugo:oliv:2016}  and references therein.
The methods use known $\sigma$ though an upper bound
or upper confidence limit for $\sigma$ (as in GAIL) could be used.

We could split our $R$ replicates into a small number $b$
of subsets of $R/b$ replicates, and take the median of $b$
subset means.  We could then get a finite $R$ confidence
interval for $\mu$, but the ones in \cite{devr:lera:lugo:oliv:2016} 
require knowledge of $\sigma^2_n=\var(\hat\mu_n^{(r)})$ or an 
upper bound for that variance. They also have width proportional to
$\sigma_n$ while the median of means estimates in
\cite{superpolymulti,superpolyone,pan:2024:tr}
have error $o(\sigma_n)$. Therefore a median of means confidence
interval would be quite conservative.
Gobet et al.\ \cite{gobe:lera:meti:2022:tr} compare several
robust confidence interval strategies for RQMC, including median
of means, for known $\sigma_n$, but do not
declare a winning method.

\subsection*{$R$ growing with $n$}

For $N=Rn$ function evaluations, we could take $n=N^c$
and $R=N^{1-c}$ (i.e., integer values near these) for $0<c<1$
as $N\to\infty$ as studied in \cite{naka:tuff:2024}.
To remove an uninteresting complication,
they assume  that $\sigma_n>0$ for all $n$.
They develop CLTs for $\hat\mu =(1/R)\sum_{r=1}^R\hat\mu_{N/R}^{(r)}$.
A CLT along with an estimate $\hat\sigma^2_n$,
such as $\tilde s_n^2$ of \eqref{eq:rqmcsamplevar},
which satisfies
$\lim_{n\to\infty}\Pr\bigl( |\hat\sigma_n^2/\sigma^2_n-1|>\epsilon\bigr)=0$ 
for any $\epsilon>0$ provides an asymptotically valid confidence interval for $\mu$
via~\eqref{eq:cltci} or~\eqref{eq:stdci} (with $R$ playing the role
of the sample size in those equations).

Even though $\hat\mu_{N/R}^{(r)}$ are IID for fixed $N$, that common 
distribution changes as $N\to\infty$ and so the CLT
has to be of the triangular array type. That requires a Lindeberg condition 
(see page 13{:}10 of \cite{naka:tuff:2024})
for which a simpler Lyapunov condition is sufficient. The Lyapunov condition is that
\begin{align}\label{eq:lyapunov}
\frac{ \bbE\bigl( |\hat\mu_{N/R}-\mu|^{2+\delta}\bigr)}{R^{\delta/2}\sigma^{2+\delta}_{N/R}}
=\frac{ \bbE\bigl( |\hat\mu_{N^c}-\mu|^{2+\delta}\bigr)}{N^{(1-c)\delta/2}\sigma^{2+\delta}_{N^c}}
\to0
\end{align}
as $N\to\infty$,  for some $\delta>0$.
A sufficient condition for~\eqref{eq:lyapunov}
is that $\bbE( |\hat\mu_n-\mu|^{2+\delta})/\sigma_n^{2+\delta}<k$
holds for some $k$ and all sufficiently large $n$.
Table 2 of \cite{naka:tuff:2024} gives upper bounds on $c$
to get a CLT
under various assumptions on the regularity of $f$ and the Lyapunov
conditions.  There are also some upper bounds under
the additional requirement that $\hat\sigma_n/\sigma_n$ converges in probability to one.

For matrix scrambling with a digital shift and a smooth
enough integrand, we have $\bbE( |\hat\mu_n-\mu|^{2+\delta})
=\Omega(n^{-3-\delta})$ by the argument in
Section~\ref{sec:rqmc}  and $\sigma^2_n=O(n^{-3+\epsilon})$.
Then
$$
\frac{\Omega(n^{-3-\delta})}{R^{\delta/2}O(n^{(-3+\epsilon)(1+\delta/2)})}
=\Omega\Bigl(\frac{n^{\delta/2-\epsilon(1+\delta/2)}}{R^{\delta/2}}\Bigr)
$$
for any $\epsilon>0$.  As a result,~\eqref{eq:lyapunov}
cannot hold for any $\delta>0$ if $R =o(n)$. 

Other RQMC methods
and other smoothness conditions
can impose less stringent requirements on $R$ for a CLT.
For example the digital nets of
\cite{faures} scrambled as in \cite{rtms} are known to yield 
a CLT for $\hat\mu$ as $n\to\infty$ for fixed
finite $R$ \cite{loh:2003}.
When UQ is the primary criterion with accuracy
secondary, then those $t=0$ nets might
be preferable to the ones of Sobol' with $t>0$
where no CLT has been proven.

% \subsection*{Random sample sizes}

%  Another randomization that has
%  been used in QMC is to randomize the sample size $n$,
%  such as by choosing it uniformly between $N/2$ and $N$ for
%  some $N>0$.  
%  It reduces the chance that the sample size $n$ we use is
%  especially bad and this can lead to a better convergence rate.
%  This approach goes back at least to Bahkvalov. It has
%  been used more recently by xxx and xxx.
%  This approach does not necessarily give $\mathbb{E}(\hat\mu)=\mu$.
%  This unbiasedness is an important ingredient in constructing
%  confidence intervals by RQMC.

\section{Conclusions}\label{sec:conc}

We have seen that uncertainty quantification for QMC estimates
is subject to a complex mix of gaps and tradeoffs and surprises.
The bracketing inequalities for smooth convex integrands
are simple to use and very easy to understand.  However
they require strong assumptions,  exhibit a severe
dimension effect and can greatly overestimate the size of the error.
The other methods with certificates also require
strong assumptions and show a dimension effect.
The methods based on NNLD and NPLD points
use more involved constructions and more complex derivations.

We may be able to get a certificate that becomes narrow
at the customary convergence rate if we know $V_{\mathrm{HK}}(f)$
or a weighted Hilbert space norm for $f$.
However such knowledge is likely to be extremely rare.
An unanchored weighted Hilbert space norm for $f$ also includes
a term  $\bigl(\int_{[0,1]^d} |\partial^{1:d}f(\bsx)|^2\rd \bsx
/\gamma_{1:d}\bigr)^{1/2}$
which can be large enough to make such an error estimate very
conservative.

When we are willing to accept a confidence interval instead of
a certificate, then more methods become available.  These
require strong assumptions such as a known bound on $f$
or a known cone to which $f$ must belong. Confidence intervals
that are still correct for any $f$ in a large class of integrands
can be very conservatively wide.

If we are willing to accept an asymptotic confidence interval,
then more choices become available.
A plain Student's $t$ confidence interval based on a modest
number $R$ of random replicates of an RQMC rule performed
well in simulations. One explanation, which needs more
study, is that $\hat\mu_n$ from at least some RQMC
methods takes on a more symmetric distribution as $n\to\infty$.

The phenomenon of $\kappa_n\to\infty$ poses extreme
difficulty for $\tilde s^2_n$ to converge to $\sigma^2_n$, which
is one of the conditions that \cite{naka:tuff:2024}
use for asymptotic confidence intervals.  That
would require $\kappa_n/R\to0$, but we want small
$R$ for accurate estimation of $\mu$.  When $\kappa_n\to\infty$
due to rare outliers, then we may not
see any of those outliers among $R$ replicates.  We will then get an
estimate $\tilde s^2$ that is far smaller than $\sigma^2_n$.
Ordinarily that would be very unfavorable for confidence
interval coverage.  However if the distribution of $\hat\mu$
is symmetric or close enough to symmetric, then the standard
interval will give reliable and even somewhat conservative
coverage despite the frequent underestimation of $\sigma_n$.

It is a pleasant surprise that some RQMC estimates tend towards
symmetry for large $n$ making it possible to get reliable
confidence intervals for $\mu$
without having a CLT or a good estimate of $\sigma^2_n$.
This clearly needs more study, to see what conditions we need
on the integrands and RQMC methods for this to happen.

\section*{Acknowledgments}

This work was supported by the National Science Foundation
under grant DMS-2152780.
I thank the people behind MCQMC 2024:
Christiane Lemieux, Ben Feng, Nathan Kirk,
Adam Kolkiewicz, Carla Daniels and Greg Preston
for organizing such a delightful meeting.
Most of the new work reported here was done
in collaboration with Pierre L'Ecuyer, Bruno Tuffin,
Marvin Nakayama, Michael Gnewuch, Peter Kritzer
and Zexin Pan.
I received helpful comments on this document
from Peter Kritzer, Marvin Nakayama, Pierre L'Ecuyer
and two anonymous reviewers.

% %%% To ensure the bibliography has the correct style please run bibtex with the spmpsci style
%%% which is included in the Springer zip file
%%% Here we assume refs.bib would be the name of the bib file containing bibliographic info
%%% You can then copy the .bbl produced, as given in the example below
%%%
%\bibliographystyle{spmpsci}
%\bibliography{refs}
%
% ---- Bibliography ----
%

\bibliographystyle{spmpsci}
\bibliography{qmc}

\begin{thebibliography}{100}
\providecommand{\url}[1]{{#1}}
\providecommand{\urlprefix}{URL }
\expandafter\ifx\csname urlstyle\endcsname\relax
  \providecommand{\doi}[1]{DOI~\discretionary{}{}{}#1}\else
  \providecommand{\doi}{DOI~\discretionary{}{}{}\begingroup
  \urlstyle{rm}\Url}\fi

\bibitem{aist:dick:2015}
Aistleitner, C., Dick, J.: Functions of bounded variation, signed measures, and
  a general {Koksma-Hlawka} inequality.
\newblock Acta Arithmetica \textbf{167}(2), 143--171 (2015)

\bibitem{andr:2024}
Andral, C.: Combining normalizing flows and quasi-{Monte Carlo}.
\newblock Tech. rep., arXiv:2401.05934 (2024)

\bibitem{atch:jaco:2025}
Atchad\'e, Y.F., Jacob, P.E.: Unbiased {Markov chain Monte Carlo}: what, why
  and how.
\newblock Tech. rep., arxiv:2406.06825 (2024)

\bibitem{aust:mack:2022}
Austern, M., Mackey, L.: Efficient concentration with {Gaussian} approximation.
\newblock Tech. rep., arXiv:2208.09922 (2022)

\bibitem{baha:sava:1956}
Bahadur, R.R., Savage, L.J.: The nonexistence of certain statistical procedures
  in nonparametric problems.
\newblock The Annals of Mathematical Statistics \textbf{27}(4), 1115--1122
  (1956)

\bibitem{bakh:1959}
Bakhvalov, N.S.: On approximate calculation of multiple integrals.
\newblock Vestnik Moskovskogo Universiteta, Seriya Matematiki, Mehaniki,
  Astronomi, Fiziki, Himii \textbf{4}, 3--18 (1959).
\newblock (In Russian)

\bibitem{benj:1983}
Benjamini, Y.: Is the $t$ test really conservative when the parent distribution
  is long-tailed?
\newblock Journal of the American Statistical Association \textbf{78}(383),
  645--654 (1983)

\bibitem{boyd:vand:2004}
Boyd, S., Vandeberghe, L.: Convex Optimization.
\newblock Cambridge University Press, Cambridge (2004)

\bibitem{brio:etal:2019}
Briol, F.X., Oates, C.J., Girolami, M., Osborne, M.A., Sejdinovic, D.:
  Probabilistic integration.
\newblock Statistical Science \textbf{34}(1), 1--22 (2019)

\bibitem{broo:gelm:jone:meng:2011}
Brooks, S., Gelman, A., Jones, G., Meng, X.L.: Handbook of {Markov chain Monte
  Carlo}.
\newblock CRC press, Boca Raton, FL (2011)

\bibitem{chen:2011}
Chen, S.: Consistency and convergence rate of {Markov chain quasi Monte Carlo}
  with examples.
\newblock Ph.D. thesis, Stanford University (2011)

\bibitem{qmcmetrocts}
Chen, S., Dick, J., Owen, A.B.: Consistency of {Markov} chain quasi-{Monte
  Carlo} on continuous spaces.
\newblock Annals of Statistics \textbf{39}(2), 673--701 (2011)

\bibitem{chen:sriv:trav:2014}
Chen, W., Srivastav, A., Travaglini, G. (eds.): A Panorama of Discrepancy
  Theory.
\newblock Springer, Cham, Switzerland (2014)

\bibitem{chen:1967}
Chentsov, N.N.: Pseudorandom numbers for modelling {Markov} chains.
\newblock Computational Mathematics and Mathematical Physics \textbf{7},
  218--233 (1967)

\bibitem{choi:etal:2024}
Choi, S.C.T., Hickernell, F.J., Rathinavel, J., McCourt, M.J., Sorokin, A.G.:
  {QMCPy}: A quasi-{Monte Carlo} {Python} library.
\newblock \url{https://pypi.org/project/qmcpy/} (2024)

\bibitem{chop:papa:2020}
Chopin, N., Papaspiliopoulos, O.: An introduction to sequential {Monte Carlo}.
\newblock Springer (2020)

\bibitem{cres:1980}
Cressie, N.: Relaxing assumptions in the one sample t-test.
\newblock Australian Journal of Statistics \textbf{22}(2), 143--153 (1980)

\bibitem{cruz2002sharp}
Cruz-Uribe, D., Neugebauer, C.J.: Sharp error bounds for the trapezoidal rule
  and {Simpson's} rule.
\newblock J. Inequal. Pure Appl. Math \textbf{3}(4), 1--22 (2002)

\bibitem{davrab}
Davis, P.J., Rabinowitz, P.: Methods of Numerical Integration (2nd Ed.).
\newblock Academic Press, San Diego (1984)

\bibitem{davi:hink:1997}
Davison, A.C., Hinkley, D.V.: Bootstrap Methods and Their Application.
\newblock Cambridge University Press, Cambridge (1997)

\bibitem{devr:1986}
Devroye, L.: Non-uniform Random Variate Generation.
\newblock Springer (1986)

\bibitem{devr:lera:lugo:oliv:2016}
Devroye, L., Lerasle, M., Lugosi, G., Oliveira, R.I.: Sub-{Gaussian} mean
  estimators.
\newblock The Annals of Statistics \textbf{44}(6), 2695--2725 (2016)

\bibitem{dici:efro:1996}
DiCiccio, T.J., Efron, B.: Bootstrap confidence intervals.
\newblock Statistical science \textbf{11}(3), 189--228 (1996)

\bibitem{dick:2009}
Dick, J.: The decay of the {Walsh} coefficients of smooth functions.
\newblock Bulletin of the Australian Mathematical Society \textbf{80}(3),
  430--453 (2009)

\bibitem{dick:2011}
Dick, J.: Higher order scrambled digital nets achieve the optimal rate of the
  root mean square error for smooth integrands.
\newblock The Annals of Statistics \textbf{39}(3), 1372--1398 (2011)

\bibitem{dick:krit:2006}
Dick, J., Kritzer, P.: A best possible upper bound on the star discrepancy of
  (t, m, 2)-nets.
\newblock Monte Carlo Methods and Applications \textbf{12}(1), 1--17 (2006)

\bibitem{dick:kuo:sloa:2013}
Dick, J., Kuo, F.Y., Sloan, I.H.: High-dimensional integration: the
  {quasi-Monte Carlo} way.
\newblock Acta Numerica \textbf{22}, 133--288 (2013)

\bibitem{dick:pill:2010}
Dick, J., Pillichshammer, F.: Digital sequences, discrepancy and quasi-{Monte
  Carlo} integration.
\newblock Cambridge University Press, Cambridge (2010)

\bibitem{dimo:2008}
Dimov, I.T.: Monte Carlo methods for applied scientists.
\newblock World Scientific, Singapore (2008)

\bibitem{dobk:epps:mitc:1996}
Dobkin, D.P., Eppstein, D., Mitchell, D.P.: Computing the discrepancy with
  applications to supersampling patterns.
\newblock ACM Transactions on Graphics (TOG) \textbf{15}(4), 354--376 (1996)

\bibitem{du:he:2024}
Du, J., He, Z.: Unbiased {Markov chain quasi-Monte Carlo for Gibbs} samplers.
\newblock arXiv preprint arXiv:2403.04407  (2024)

\bibitem{efro:1969}
Efron, B.: Student's t-test under symmetry conditions.
\newblock Journal of the American Statistical Association \textbf{64}(328),
  1278--1302 (1969)

\bibitem{efro:1987}
Efron, B.: Better bootstrap confidence intervals (with discussion).
\newblock Journal of the American statistical Association \textbf{82}(397),
  171--185 (1987)

\bibitem{efrontibs}
Efron, B.M., Tibshirani, R.J.: An Introduction to the Bootstrap.
\newblock Chapman and Hall (1993)

\bibitem{esar:pros:walk:1967}
Esary, J.D., Proschan, F., Walkup, D.W.: Association of random variables, with
  applications.
\newblock The Annals of Mathematical Statistics \textbf{38}(5), 1466--1474
  (1967)

\bibitem{faures}
Faure, H.: Discr\'epance de suites associ\'ees \`a un syst\`eme de num\'eration
  (en dimension $s$).
\newblock Acta Arithmetica \textbf{41}, 337--351 (1982)

\bibitem{frec:1910}
Fr\'echet, M.: Extension au cas des int\'egrales multiples d'une d\'efinition
  de l'int\'egrale due \`a {Stieltjes}.
\newblock Nouvelles Annales de Math\'ematiques \textbf{10}, 241--256 (1910)

\bibitem{gaba:1967}
Gabai, H.: On the discrepancy of certain sequences mod 1.
\newblock Illinois Journal of Mathematics \textbf{11}(1), 1--12 (1967)

\bibitem{gnew:krit:owen:pan:2024}
Gnewuch, M., Kritzer, P., Owen, A.B., Pan, Z.: Computable error bounds for
  {quasi-Monte Carlo} using points with non-negative local discrepancy.
\newblock Information and Inference: A Journal of the IMA \textbf{13}(3),
  iaae021 (2024)

\bibitem{gobe:lera:meti:2022:tr}
Gobet, E., Lerasle, M., M\'etivier, D.: Mean estimation for randomized quasi
  {Monte Carlo} method.
\newblock Tech. rep., hal-03631879 (2022)

\bibitem{goda:krie:2024:tr}
Goda, T., Krieg, D.: A simple universal algorithm for high-dimensional
  integration.
\newblock Tech. rep., arXiv:2411.19164 (2024)

\bibitem{goda:lecu:2022}
Goda, T., L'Ecuyer, P.: Construction-free median {quasi-Monte Carlo} rules for
  function spaces with unspecified smoothness and general weights.
\newblock SIAM Journal on Scientific Computing \textbf{44}(4), A2765--A2788
  (2022)

\bibitem{gues:schm:2004}
Guessab, A., Schmeisser, G.: Convexity results and sharp error estimates in
  approximate multivariate integration.
\newblock Mathematics of computation \textbf{73}(247), 1365--1384 (2004)

\bibitem{hall:1986}
Hall, P.G.: On the bootstrap and confidence intervals.
\newblock The Annals of Statistics pp. 1431--1452 (1986)

\bibitem{hall:1988}
Hall, P.G.: Theoretical comparisons of bootstrap confidence intervals.
\newblock The Annals of Statistics \textbf{16}(3), 927--953 (1988)

\bibitem{hall:1992}
Hall, P.G.: The Bootstrap and Edgeworth Expansion.
\newblock Springer, New York (1992)

\bibitem{halt:2005}
Halton, J.: Quasi-probability: Why {quasi-Monte-Carlo} methods are
  statistically valid and how their errors can be estimated statistically.
\newblock {Monte Carlo} methods and applications \textbf{11}(3), 203--350
  (2005)

\bibitem{hick:2002}
Hickernell, F.J.: Obtaining ${O(N^{-2+\epsilon})}$ convergence for lattice
  quadrature rules.
\newblock In: K.T. Fang, F.J. Hickernell, H.~Niederreiter (eds.) {M}onte
  {C}arlo and Quasi-{M}onte {C}arlo Methods 2000, pp. 274--289.
  Springer-Verlag, Berlin (2002)

\bibitem{hick:2014}
Hickernell, F.J.: {Koksma-Hlawka} inequality.
\newblock Wiley StatsRef: Statistics Reference Online  (2014)

\bibitem{gailmc}
Hickernell, F.J., Jiang, L., Liu, Y., Owen, A.B.: Guaranteed conservative fixed
  width confidence intervals via {Monte Carlo} sampling.
\newblock In: J.~Dick, F.Y. Kuo, G.~Peters, I.H. Sloan (eds.) {Monte Carlo} and
  Quasi-{Monte Carlo} Methods 2012, pp. 105--128. Springer, Berlin (2013)

\bibitem{hick:ruga:2016}
Hickernell, F.J., Rugama, L.A.J.: Reliable adaptive cubature using digital
  sequences.
\newblock In: R.~Cools, D.~Nuyens (eds.) Monte Carlo and Quasi-Monte Carlo
  Methods 2014, pp. 367--383. Springer, Cham, Switzerland (2016)

\bibitem{hoef:1963}
Hoeffding, W.: Probability inequalities for sums of bounded random variables.
\newblock Journal of the American Statistical Association \textbf{58}, 13--30
  (1963)

\bibitem{jain:etal:2025}
Jain, A., Hickernell, F.J., Owen, A.B., Sorokin, A.G.: Empirical {Bernstein}
  and betting confidence intervals for randomized {quasi-Monte Carlo}.
\newblock Tech. rep., arXiv:2504.18677 (2025)

\bibitem{jime:hick:2016}
Jim{\'e}nez~Rugama, L.A., Hickernell, F.J.: Adaptive multidimensional
  integration based on rank-1 lattices.
\newblock In: Monte Carlo and Quasi-Monte Carlo Methods: MCQMC, Leuven,
  Belgium, April 2014, pp. 407--422. Springer (2016)

\bibitem{joe:kuo:2008}
Joe, S., Kuo, F.Y.: Constructing {Sobol'} sequences with better two-dimensional
  projections.
\newblock SIAM Journal on Scientific Computing \textbf{30}(5), 2635--2654
  (2008)

\bibitem{kats:nova:petr:1996}
Katscher, C., Novak, E., Petras, K.: Quadrature formulas for multivariate
  convex functions.
\newblock Journal of Complexity \textbf{12}(1), 5--16 (1996)

\bibitem{lecu:lemi:2000}
L'Ecuyer, P., Lemieux, C.: Variance reduction via lattice rules.
\newblock Management Science \textbf{46}(9), 1214--1235 (2000)

\bibitem{lecu:lemi:2002}
L'Ecuyer, P., Lemieux, C.: A survey of randomized quasi-{M}onte {C}arlo
  methods.
\newblock In: M.~Dror, P.~L'Ecuyer, F.~Szidarovszki (eds.) Modeling
  Uncertainty: An Examination of Stochastic Theory, Methods, and Applications,
  pp. 419--474. Kluwer Academic Publishers (2002)

\bibitem{lecu:mung:2026}
L'Ecuyer, P., Munger, D.: Algorithm 958: {Lattice Builder}: A general software
  tool for constructing rank-1 lattice rules.
\newblock ACM Transactions on Mathematical Software \textbf{42}(2), Article 15
  (2016)

\bibitem{lecu:mung:tuff:2010}
L'Ecuyer, P., Munger, D., Tuffin, B.: On the distribution of integration error
  by randomly-shifted lattice rules.
\newblock Electronic Journal of Statistics \textbf{4}, 950--993 (2010)

\bibitem{ci4rqmc}
L'Ecuyer, P., Nakayama, M.K., Owen, A.B., Tuffin, B.: Confidence intervals for
  randomized quasi-{Monte Carlo} estimators.
\newblock In: C.G. Corlu, S.R. Hunter, H.~Lam, B.S. Onggo, J.~Shortle,
  B.~Biller (eds.) 2023 Winter Simulation Conference (WSC), pp. 445--456. IEEE
  (2023)

\bibitem{lecu:sanv:2010}
L'Ecuyer, P., Sanvido, C.: Coupling from the past with randomized quasi-{Monte
  Carlo}.
\newblock Mathematics and Computers in Simulation \textbf{81}(3), 476--489
  (2010)

\bibitem{lecu:sima:2007}
L'Ecuyer, P., Simard, R.: {TestU01}: A {C} library for empirical testing of
  random number generators.
\newblock ACM Transactions on Mathematical Software (TOMS) \textbf{33}(4),
  1--40 (2007)

\bibitem{levi:1999}
Levin, M.B.: Discrepancy estimates of completely uniformly distributed and
  pseudo-random number sequences.
\newblock International Mathematics Research Notices pp. 1231--1251 (1999)

\bibitem{liu:2024}
Liu, S.: {Langevin} quasi-{Monte Carlo}.
\newblock Advances in Neural Information Processing Systems \textbf{36} (2024)

\bibitem{liu:2024:transport}
Liu, S.: Transport quasi-{Monte Carlo}.
\newblock Tech. rep., arXiv:2412.16416 (2024)

\bibitem{loh:2003}
Loh, W.L.: On the asymptotic distribution of scrambled net quadrature.
\newblock Annals of Statistics \textbf{31}(4), 1282--1324 (2003)

\bibitem{lyne:1983}
Lyness, J.N.: When not to use an automatic quadrature routine.
\newblock SIAM Review \textbf{25}, 63--87 (1983)

\bibitem{mato:1998:2}
Matou\v{s}ek, J.: On the {L$^2$}--discrepancy for anchored boxes.
\newblock Journal of Complexity \textbf{14}, 527--556 (1998)

\bibitem{maur:pont:2009}
Maurer, A., Pontil, M.: Empirical {Bernstein} bounds and sample variance
  penalization.
\newblock Tech. rep., arXiv:0907.3740 (2009)

\bibitem{beyondanova}
Miller, R.G.: Beyond ANOVA, Basics of Applied Statistics.
\newblock Wiley (1986)

\bibitem{naka:tuff:2024}
Nakayama, M.K., Tuffin, B.: Sufficient conditions for central limit theorems
  and confidence intervals for randomized {quasi-Monte Carlo} methods.
\newblock ACM Transactions on Modeling and Computer Simulation \textbf{34}(3),
  1--38 (2024)

\bibitem{nied:1978}
Niederreiter, H.: Quasi-{Monte Carlo} methods and pseudo-random numbers.
\newblock Bulletin of the American Mathematical Society \textbf{84}(6),
  957--1041 (1978)

\bibitem{nied92}
Niederreiter, H.: Random Number Generation and Quasi-{Monte Carlo} Methods.
\newblock S.I.A.M., Philadelphia, PA (1992)

\bibitem{owen:smallci}
Owen, A.B.: Small sample central confidence intervals for the mean.
\newblock Tech. Rep. 302, Stanford University, Department of Statistics (1988).
\newblock \url{https://purl.stanford.edu/mz765np4744}, Accessed 14th August
  2023

\bibitem{rtms}
Owen, A.B.: Randomly permuted $(t,m,s)$-nets and $(t,s)$-sequences.
\newblock In: H.~Niederreiter, P.J.S. Shiue (eds.) Monte Carlo and Quasi-Monte
  Carlo Methods in Scientific Computing, pp. 299--317. Springer-Verlag, New
  York (1995)

\bibitem{variation}
Owen, A.B.: Multidimensional variation for quasi-{Monte Carlo}.
\newblock In: J.~Fan, G.~Li (eds.) International Conference on Statistics in
  honour of Professor Kai-Tai Fang's 65th birthday (2005)

\bibitem{warnockhaltonqse}
Owen, A.B.: On the {Warnock-Halton} quasi-standard error.
\newblock Monte Carlo Methods \& Applications \textbf{12}(1) (2006)

\bibitem{practicalqmc}
Owen, A.B.: Practical Quasi-Monte Carlo Integration.
\newblock \url{https://artowen.su.domains/mc/practicalqmc.pdf} (2023)

\bibitem{hallformula-tr}
Owen, A.B.: Coverage errors for {Student's $t$} confidence intervals comparable
  to those in {Hall} (1988).
\newblock Tech. rep., arxiv:2501.07645 (2025)

\bibitem{thelogs}
Owen, A.B., Pan, Z.: Where are the logs?
\newblock In: Z.~Botev, A.~Keller, B.~Tuffin (eds.) Advances in Modeling and
  Simulation: Festschrift for Pierre L'Ecuyer, pp. 381--400. Springer (2022)

\bibitem{qmcmetro}
Owen, A.B., Tribble, S.D.: A quasi-{Monte Carlo} {Metropolis} algorithm.
\newblock Proceedings of the National Academy of Sciences \textbf{102}(25),
  8844--8849 (2005)

\bibitem{pan:2024:tr}
Pan, Z.: Automatic optimal-rate convergence of randomized nets using
  median-of-means.
\newblock Tech. rep., arXiv:2411.01397 (2024)

\bibitem{superpolyone}
Pan, Z., Owen, A.B.: Super-polynomial accuracy of one dimensional randomized
  nets using the median-of-means.
\newblock Mathematics of Computation \textbf{92}(340), 805--837 (2023)

\bibitem{lowskewness}
Pan, Z., Owen, A.B.: Skewness of a randomized {quasi-Monte Carlo} estimate.
\newblock Tech. rep., arXiv:2405.06136 (2024)

\bibitem{superpolymulti}
Pan, Z., Owen, A.B.: Super-polynomial accuracy of multidimensional randomized
  nets using the median-of-means.
\newblock Mathematics of Computation  (2024)

\bibitem{papa:1993}
Papageorgiou, A.: Integration of monotone functions of several variables.
\newblock Journal of Complexity \textbf{9}(2), 252--268 (1993)

\bibitem{prop:wils:1996}
Propp, J.G., Wilson, D.B.: Exact sampling with coupled {Markov} chains and
  applications to statistical mechanics.
\newblock Random Structures and Algorithms \textbf{9}, 223--252 (1996)

\bibitem{reze:moha:2015}
Rezende, D., Mohamed, S.: Variational inference with normalizing flows.
\newblock In: International conference on machine learning, pp. 1530--1538.
  PMLR (2015)

\bibitem{shao:wang:2013}
Shao, Q.M., Wang, Q.: Self-normalized limit theorems: {A} survey.
\newblock Probability Surveys \textbf{10}, 69--93 (2013)

\bibitem{shev:2011}
Shevtsova, I.: On the absolute constants in the {Berry-Esseen} type
  inequalities for identically distributed summands.
\newblock Tech. rep., arXiv:1111.6554 (2011)

\bibitem{sloa:joe:1994}
Sloan, I.H., Joe, S.: Lattice Methods for Multiple Integration.
\newblock Oxford Science Publications, Oxford (1994)

\bibitem{sloa:wozn:1998}
Sloan, I.H., Wozniakowski, H.: When are quasi-{Monte Carlo} algorithms
  efficient for high dimensional integration?
\newblock Journal of Complexity \textbf{14}, 1--33 (1998)

\bibitem{sobo:1967}
Sobol', I.M.: The use of {Haar} series in estimating the error in the
  computation of infinite-dimensional integrals.
\newblock Dokl. Akad. Nauk SSSR \textbf{8}(4), 810--813 (1967)

\bibitem{sorokin2022bounding}
Sorokin, A.G., Rathinavel, J.: On bounding and approximating functions of
  multiple expectations using {Quasi-Monte Carlo}.
\newblock In: International Conference on Monte Carlo and Quasi-Monte Carlo
  Methods in Scientific Computing, pp. 583--599. Springer (2022)

\bibitem{tong:etal:2022}
Tong, X., Choi, S.C.T., Ding, Y., Hickernell, F.J., Jiang, L., Rugama, L.A.J.,
  Rathinavel, J., Zhang, K., Zhang, Y., Zhou, X.: Guaranteed automatic
  integration library ({GAIL}): An open-source {MATLAB} library for function
  approximation, optimization, and integration.
\newblock Journal of Open Research Software \textbf{10}(1) (2022)

\bibitem{trib:2007}
Tribble, S.D.: {Markov chain Monte Carlo} algorithms using completely uniformly
  distributed driving sequences.
\newblock Ph.D. thesis, Stanford University (2007)

\bibitem{qmcmetro2}
Tribble, S.D., Owen, A.B.: Construction of weakly {CUD} sequences for {MCMC}
  sampling.
\newblock Electronic Journal of Statistics \textbf{2}, 634--660 (2008)

\bibitem{warn:2001}
Warnock, T.: Effective error estimates for quasi-{Monte Carlo} computations.
\newblock Tech. Rep. LA-UR-01-1950, Los Alamos National Labs (2001)

\bibitem{warn:2002}
Warnock, T.: Effective error estimates for quasi-{Monte Carlo} computations.
\newblock http://lib-www.lanl.gov/la-pubs/00367143.pdf (2002)

\bibitem{waud:ramd:2024}
Waudby-Smith, I., Ramdas, A.: Estimating means of bounded random variables by
  betting.
\newblock Journal of the Royal Statistical Society Series B: Statistical
  Methodology \textbf{86}(1), 1--27 (2024)

\bibitem{yosh:2017}
Yoshiki, T.: Bounds on {Walsh} coefficients by dyadic difference and a new
  {Koksma-Hlawka} type inequality for quasi-{Monte Carlo} integration.
\newblock Hiroshima Mathematical Journal \textbf{47}(2), 155--179 (2017)

\end{thebibliography}
\end{document}